\documentclass[journal]{IEEEtran}
\usepackage{cite, graphicx, subfigure,url,setspace,algorithm}
\usepackage{amsmath,amssymb}
\usepackage{bm}
\usepackage{subfigure}
\usepackage{multirow}
\usepackage{tikz}

\newtheorem{proposition} {Proposition}
\newtheorem{remark}{Remark}

\newtheorem{corollary}{Corollary}
\begin{document}

\title{A Computationally Optimal Randomized Proper Orthogonal Decomposition Technique}

\author{Dan Yu, Suman Chakravorty
\thanks{D. Yu is with the Department
of Aerospace Engineering, Texas A\&M University, College Station, TX, 77840 USA e-mail: (yudan198811@hotmail.com).}
\thanks{S. Chakravorty is with the Department of  Aerospace Engineering, Texas A\&M University, College Station, TX 77840 USA e-mail: (schakrav@tamu.edu)}}


\maketitle


\begin{abstract}
In this paper, we consider the model reduction problem of large-scale systems, such as systems obtained through the discretization of partial differential equations. We propose a computationally optimal randomized proper orthogonal decomposition (RPOD$^*$) technique to obtain the reduced order model by  perturbing the primal and adjoint system using Gaussian white noise. We show that the computations  required by the RPOD$^*$ algorithm is orders of magnitude cheaper when compared to the balanced proper orthogonal decomposition (BPOD) algorithm and BPOD output projection algorithm while the performance of the RPOD$^*$ algorithm is much better than BPOD output projection algorithm. It is computationally optimal in the sense that a minimal number of snapshots is needed. We also relate the RPOD$^*$ algorithm to random projection algorithms. The method is tested on two advection-diffusion problems. 
\end{abstract}

\section{Introduction}

In this paper, we are interested in the model reduction of large scale systems such as those  governed by partial differential equations (PDE). The dimension of the system is large due to the discretization of the PDEs. For instance, consider the atmospheric dispersion of an air pollutant \cite{stockie}. The emission of the contaminants on the ground level is shown in Fig. \ref{Emission} with four point sources labeled from S1 to S4.

This is a three dimensional problem, and after discretizing the PDE, the dimension of the system is $10^6$.  Therefore, we are interested in constructing a reduced order model (ROM) that can capture the input/output characteristics of the large model such that this ROM can then be used by a filtering algorithm for updating the states of the field, such as the Kalman filter. Also,  the actuators and sensors could be placed anywhere in this field, which leads to a model reduction problem of a large-scale system with a large number of inputs/outputs.  There are two popular contemporary model reduction techniques that have been studied in the past few decades: Principal component analysis (PCA) and randomization algorithms.

\begin{figure}[t]
\centering
\includegraphics[width=1.67in]{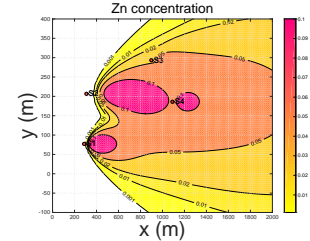}
\caption{Air pollutant problem}
\label{Emission}
\end{figure}

Among the PCA algorithms, Balanced Proper Orthogonal Decomposition (BPOD) \cite{willcox, rowley1} based on the balanced truncation \cite{moore}  and the snapshot Proper Orthogonal Decomposition (POD) technique  \cite{pod2} has been widely used. Balancing transformations are constructed  using  the impulse responses of both the primal and adjoint system, and hence, the most controllable and observable modes can be kept in the ROM. In 1978, Kung \cite{kung}  presented a new model reduction algorithm in conjunction with the singular value decomposition (SVD) technique, and the eigensystem realization algorithm (ERA) \cite{juang1985} was developed based on this technique. The BPOD is equivalent to the ERA procedure \cite{rowley4}, and forms the Hankel matrix using the primal and adjoint system simulations as opposed to the input-output data as in ERA. More recently, there has been work on obtaining information regarding the dominant modes of the system based on the snapshot POD, followed by an eigendecomposition of an approximating linear operator,  called the dynamic mode decomposition (DMD) \cite{schmid, rowleynew}. In \cite{apod1, apod2}, an adaptive POD algorithm based on the snapshot POD algorithm is introduced to recursively revise the locally valid ROMs. When there are new snapshots collected, decisions are made if the basis functions need to be updated. Both the DMD and adaptive POD algorithm are applicable for the nonlinear systems. 

The primary drawback of the BPOD and ERA  is that for a large scale system, such as that obtained by discretizing a PDE,  with a large number of inputs/outputs, the computational burden incurred is very high. There are two main parts to the computation: the first is to collect datasets from computationally expensive primal and adjoint simulation in order to generate the Hankel matrix. The second part is to solve the SVD problem for the resulting Hankel matrix.

To reduce the computational cost of BPOD, improved algorithms  have been proposed. For example, \cite{rowley1} proposed an output projection method to address the problem when the number of outputs is large. The outputs are projected onto a small subspace via an orthogonal projection $P_s$ that minimizes the error between the full impulse response and the projected impulse response. However, the method cannot make any claim regarding the closeness of the solution to one that is obtained from the full Hankel matrix, and is still faced with a very high computational burden when both the numbers of inputs  and outputs are large. There have also been methods proposed  \cite{rowleytail} to reduce the number of snapshots, however, the primary problem regarding large number of inputs/outputs remains the same.

There are two major classes of randomization algorithms used for low-rank matrix approximations and factorizations: random sampling algorithms and random projection algorithms. For a large scale matrix $H$, random sampling algorithms construct a rank $k$ approximation matrix $\hat{H}$ by choosing and rescaling some columns of $H$ according to certain sampling probabilities \cite{MC_SVD}, so the error satisfies
$\| H - \hat{H} \|_F \leq \| H - H_{(k)} \|_F + \epsilon \| H \|_F,$
with high probability, where $H_{(k)}$ is a best rank $k$ approximation of $H$, $\epsilon$ is a specified tolerance, and  $\| H \|_F$ denotes the Frobenius norm of H. The improved algorithm proposed in \cite{mohoney2} is to sample some columns according to leverage scores, where the leverage scores are calculated by performing the SVD of $H$, so that the error satisfies
$\| H  - \hat{H} \|_F \leq (1 + \epsilon) \|H - H_{(k)} \|_F,$
with high probability. A direct application of both algorithms would require the full Hankel matrix to be constructed, however, such a  construction of the Hankel matrix is  computationally prohibitive when the number of inputs/outputs is large. Further, the leverage scores are calculated by performing the SVD of the Hankel matrix, which is  also computationally prohibitive owing to the size of the problem.

In random projection  method \cite{randproj1},  the large matrix $H$ is projected on to an orthonormal basis $Q$ such that the error satisfies 
$\| H - QQ' H \| \leq (1 + \epsilon ) \| H - H_{(k)} \|$ with high probability, where $\| H \|$ denotes the spectral 2-norm of $H$, $(.)'$ denotes the conjugate transpose of $(.)$. A  Gaussian test matrix $\Omega$ is generated, and the orthonormal basis $Q$ is constructed by performing a QR factorization of the matrix product $H \Omega$. The bottleneck of this algorithm remains, as above, the construction of the full Hankel matrix, which is prohibitively expensive.

We had introduced an RPOD algorithm in \cite{acc2015} that randomly chooses a subset  of the input/output trajectories.  A sub-Hankel matrix is constructed using these sampled trajectories, which is then used to form a ROM in the usual BPOD fashion. 
The Markov parameters of the  ROM constructed using the sub-Hankel matrix were shown to be close to the Markov parameters of the full order system with high probability. We proved that  a lower bound exists for the number of the input/output trajectories that need to be sampled. The major problem associated with this algorithm is that different choices of the sampling algorithms would lead to different  lower bounds, and the choice of a good sampling algorithm other than the uniform distribution is unclear. 

In this paper, we propose the RPOD$^*$ algorithm which is closely related to the random projection algorithm. In the RPOD$^*$ algorithm, we perturb the primal and adjoint system with Gaussian white noise, and we prove that similar to the BPOD algorithm, the controllable and observable modes are retained in the ROM.  The Markov parameters of the ROM constructed using RPOD$^*$ are shown to be close to the Markov parameters of the full order system, while the error is  bounded. The main contribution of the RPOD$^*$ method is that it is computational orders of magnitudes less expensive when compared to the BPOD/ERA and randomization algorithms for a large-scale system with a large number of inputs/outputs.   The RPOD$^*$ algorithm can be viewed as applying the random projection on the full Hankel matrix $H$ twice without constructing the full Hankel matrix $H$, i.e., $\hat{H} = \Omega_2' Z' X \Omega_1 = \Omega_2' H \Omega_1$, where $\Omega_1, \Omega_2$ are two random projection matrices, and $Z, X$ are the usual impulse response matrices of the adjoint and primal system. However, we actually only generate $Z \Omega_2$ and $X \Omega_1$ which can be constructed from a single white noise perturbed trajectory each of the adjoint and primal system respectively, and thus, are orders of magnitude smaller in size than the impulse responses $Z$ and $X$. Thus, the computational cost to generate the Hankel matrix and to solve the SVD problem is saved by orders of magnitude. We believe that it is the most computational efficient POD algorithm. In practice, the RPOD$^*$ algorithm can be solved in real-time.

The rest of the paper is organized as follows. In Section \ref{formulation}, the problem is formulated, and in Section \ref{BPOD_sec}, we review the BPOD algorithm and  illustrate in a simplified fashion how to relate the BPOD ROM to the controllable and observable modes of the system. The RPOD$^*$ algorithm is proposed in Section \ref{RPOD_body}, and the formal proofs and results are shown. Also, we discuss some implementation problems in this section. In Section \ref{output}, we compare the RPOD$^*$ algorithm with BPOD, random projection and BPOD output projection algorithm. In Section \ref{example}, we provide computational results comparing the RPOD$^*$ with the BPOD/BPOD output projection for a one dimensional heat transfer problem and a three dimensional atmospheric dispersion problem. 
\section{Problem Formulation}\label{formulation}
Consider a stable linear input-output system:
\begin{eqnarray}\label{primal}
x_k = A x_{k-1} + B u_k, \nonumber \\
y_k = C x_k,
\end{eqnarray}
where $x_k \in \Re^N$,  $u_k \in \Re^p$, $y_k \in \Re^q$ are the states, inputs, and outputs at discrete time instant $t_k$ respectively. Assume that $A, B, C$ matrices are real. 

The adjoint system is defined as:
\begin{eqnarray}\label{adjoint}
z_k  = A' z_{k-1} + C' v_k, w_k = B' z_k,
\end{eqnarray}
where $z_k \in \Re^N$, $w_k \in \Re^p$ is the state and output of the adjoint system at time $t_k$ respectively,  $v_k \in \Re^q$. 

\textbf{Definition 1.} The Markov parameters of the system is defined as $C A^i B, i =1, \cdots$.

\textbf{Assumption 1.}  Assume that $A$ is diagonalizable and stable. 

Under Assumption 1, let,
\begin{eqnarray}\label{eigendecom}
A = V \Lambda U',
\end{eqnarray}
where $\Lambda$  are the eigenvalues, $( V, U )$ are the corresponding right and left eigenvectors. 

\textbf{Definition 2.} A mode $(\Lambda_i, U_i, V_i)$  is not controllable if $U_i' B = 0$, and is not observable if $C V_i = 0$, where $(\Lambda_i, V_i, U_i)$ is the $i^{th}$ eigenvalue-eigenvector pair.  

We partition the eigenvalues and eigenvectors $(\Lambda, V, U)$ into four parts:
\begin{eqnarray}\label{partition}
A = \begin{pmatrix} V_{co}' \\ V_{c \bar{o}}' \\ V_{\bar{c} o}' \\ V_{\bar{c} \bar{o}}' \end{pmatrix}' \begin{pmatrix} \Lambda_{co} & & & \\ & \Lambda_{c \bar{o}} & & \\ & & \Lambda_{\bar{c} o} & \\ & & & \Lambda_{\bar{c} \bar{o}} \end{pmatrix} \begin{pmatrix} U_{co}' \\ U_{c \bar{o}}' \\ U_{\bar{c} o}' \\ U_{\bar{c} \bar{o}}' \end{pmatrix},
\end{eqnarray}
where $ (\Lambda_{co}, V_{co}, U_{co})$ are the controllable and observable modes, $(\Lambda_{c \bar{o}}, V_{c \bar{o}}, U_{c \bar{o}})$ are the controllable but unobservable modes, $(\Lambda_{\bar{c} o}, V_{\bar{c} o}, U_{\bar{c} o})$ are the uncontrollable but observable modes, and $(\Lambda_{\bar{c} \bar{o}}, V_{\bar{c} \bar{o}}, U_{\bar{c}\bar{o}})$ are the uncontrollable and unobservable modes.

In this paper, we consider the model reduction problem for large-scale systems with a large number of inputs/outputs. The goal is to construct an ROM such that the outputs of the ROM $y_r$ are close to the outputs of the full order system $y$, i.e., $| y - y_r|$  is small. 

Denote the number of the controllable and observable modes is exactly $l$ throughout the paper. First, we summarize all the assumptions made in this paper.
\begin{itemize}
\item A1. $A$ is stable and diagonalizable. 
\item A2. $l \ll N$.
\item A3. $ U_{\bar{c} o }' B = 0, U_{\bar{c}\bar{o}}' B = 0, CV_{c \bar{o}} = 0, CV_{\bar{c}\bar{o}} = 0$.
\item A4. $U_{\bar{c} o }' B = \epsilon C_1, U_{\bar{c}\bar{o}}' B = \epsilon C_2, CV_{c \bar{o}} = \epsilon C_3, CV_{\bar{c}\bar{o}} = \epsilon C_4$, where $\epsilon$ is a small number, $C_1, C_2, C_3, C_4$ are constant matrices. And if $\|U_{co}' B \| = O(\| C_5 \|), \|CV_{co} \| = O(\| C_5 \|)$, then $\|C_1\|, \|C_2\|, \|C_3\|, \|C_4\| = O(\|C_5\|)$, where $C_5$ is a constant matrix.
\end{itemize}

The structure of the paper is summarized in Figure \ref{flowchart}. 

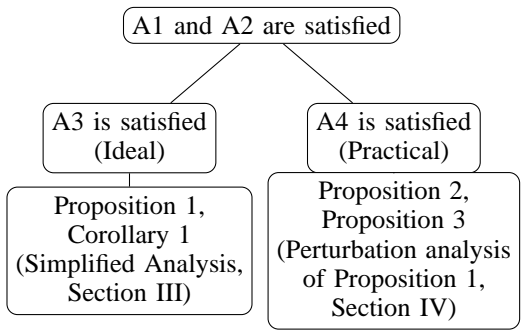
\begin{figure}[bp]
\centering
\begin{tikzpicture}[sibling distance=10em,
  every node/.style = {shape=rectangle, rounded corners,
    draw, align=center,
    top color=white, bottom color=white}]]
  \node {A1 and A2 are  satisfied}
    child { node {A3 is satisfied \\ (Ideal)}
             child{ node { Proposition 1, \\Corollary 1 \\(Simplified Analysis,\\Section \ref{BPOD_sec})} 
           }}
    child { node {A4 is  satisfied \\(Practical)}
      child { node {Proposition 2, \\Proposition 3 \\(Perturbation analysis \\ of Proposition 1, \\Section \ref{RPOD_body})}}};      
\end{tikzpicture}
\caption{Structure of the Paper}\label{flowchart}
\end{figure}

\textbf{Discussion on Assumptions.} For most of the practical applications we consider, A1 is satisfied. A2 needs to be satisfied for the system to have an ROM, this assumption is  typically satisfied for  a large-scale system. It should be noticed that from Definition 2, assumption A3 is the statement of controllability/observability of the different modes of the system. However, in practice, $U_{\bar{c} o}'B, CV_{c \bar{o}}$ are never exactly zero, and hence, in assumption A4, we assume that $ \| U_{\bar{c} o }' B \| \propto \epsilon , \| CV_{\bar{c} o} \| \propto \epsilon$, where $\epsilon$ is small. 
\section{Simplified Analysis} \label{BPOD_sec}
In this section, first, we briefly review the BPOD algorithm, and then illustrate in a simplified fashion how the transformation bases and the Markov parameters of the ROM constructed using the BPOD algorithm can be related to the controllable and observable modes $(\Lambda_{co}, V_{co}, U_{co} )$ of the system matrix $A$. The simplified  analysis is critical to understand the intuition behind the proposed RPOD$^*$ algorithm in Section \ref{RPOD_body}. 
\subsection{BPOD Algorithm}
Consider the stable linear system (\ref{primal})-(\ref{adjoint}),  let $B = [b_1, \cdots, b_p]$ where $b_i, i = 1, \cdots, p$ is the $i^{th}$ column of $B$,  and $C' = [c_1', \cdots, c_q']$, where $c_j, j = 1, \cdots, q$ is the $j^{th}$ row of $C$.  The BPOD Algorithm \cite{rowley1} is summarized in Algorithm \ref{BPOD_algo}. 
\begin{algorithm}[bt]
\begin{enumerate}
\item{Collect impulse response $X_b$ of the primal system (\ref{primal}): Use $b_i, i = 1, \cdots, p$ as initial conditions for (\ref{primal}) with $u_k = 0$. Take $m$ snapshots  at discrete times $t_1, t_2, \cdots, t_m$, and
\begin{eqnarray}\label{sp}
{X}_b = [x_1(t_1), \cdots, x_p(t_{1}), \cdots, x_1(t_m), \cdots, x_p(t_{m})],
\end{eqnarray}
 where $x_i(t_k)$ is the state snapshot at time $t_k$ with $b_i$ as the initial condition, $k = 1, 2, \cdots, m$ and $i = 1, 2, \cdots, p$.
 }
\item{Collect impulse response $Z_b$ of the adjoint system (\ref{adjoint}):
Use  $c_j', j = 1, \cdots, q$ as initial conditions for (\ref{adjoint}) with $v_k = 0$.  Take $n$ snapshots at time step $\hat{t}_1, \hat{t}_2, \cdots, \hat{t}_n$, and 
\begin{eqnarray}\label{sa}
{Z}_b = [z_1(\hat{t}_1), \cdots, z_q(\hat{t}_{1}), \cdots, z_1(\hat{t}_n), \cdots, z_q(\hat{t}_{n})],
\end{eqnarray}
where $z_j(\hat{t}_k)$ is the state snapshot of the adjoint system at time $\hat{t}_k$ with $c_j'$ as the initial condition, $k= 1, 2, \cdots, n$ and $j= 1, 2, \cdots, q$. 
}
\item {Construct Hankel matrix $H_b$: 
\begin{eqnarray}\label{h_start}
{H}_b = Z_b' {X_b}, 
\end{eqnarray}
}
\item{Solve the SVD problem of $H_b$: 
\begin{eqnarray}
{H}_b=  \begin{pmatrix} {L}_b & {L}_1 \end{pmatrix} \begin{pmatrix} {\Sigma}_b & 0 \\ 0 & \Sigma_1 \end{pmatrix} \begin{pmatrix} {R}_b' \\ {R}_1'  \end{pmatrix},
\end{eqnarray}
where ${\Sigma}_b$ contains the first $ l$ non-zero singular values, and $({L}_b, {R}_b)$ are the corresponding left and right singular vectors. $\Sigma_1$ contains the rest of the singular values.
}
\item {Construct the BPOD bases: 
\begin{eqnarray}
{T}_b = {X}_b {R}_b  {\Sigma}_b^{-1/2}, S_b = {\Sigma}_b^{-1/2} {L}_b' Z_b', 
\end{eqnarray}
 }
\item{The ROM is: 
\begin{eqnarray}\label{BPOD}
{A}_b = S_ b  A {T}_b, {B}_b = S_b B, {C}_b  = C {T}_b.
\end{eqnarray}
}
\end{enumerate}
\caption{BPOD Algorithm}\label{BPOD_algo}
\end{algorithm}

\subsection{Simplified Analysis}\label{heuristic}
First, we construct a modal BPOD ROM by projecting the BPOD bases $(T_b, S_b)$  onto the ROM eigenspace as in Algorithm \ref{BPOD_modal}.
\begin{algorithm}[bt]
\begin{enumerate}
\item Construct BPOD ROM $(A_b, B_b, C_b)$ and BPOD bases $(T_b, S_b)$ using BPOD Algorithm \ref{BPOD_algo}. 
\item Solve the eigenvalue problem of $A_b$: 
\begin{eqnarray}
A_b = P_b \Lambda_b P_b^{-1},
\end{eqnarray} 
where $\Lambda_b$ are the eigenvalues, and $P_b$ are the corresponding eigenvectors. 
\item Construct BPOD modal bases: 
\begin{eqnarray}
\Phi_b = P_b^{-1} S_b, \Psi_b = T_b P_b,
\end{eqnarray}
\item The modal ROM is:
\begin{eqnarray}
\hat{A}_b = \Phi_b A \Psi_b, \hat{B}_b = \Phi_b' B, \hat{C}_b = C \Psi_b.
\end{eqnarray}
\end{enumerate}
\caption{BPOD modal ROM Algorithm}\label{BPOD_modal}
\end{algorithm}

Under assumptions A1, A2, and A3, we have the following result. 

\begin{proposition}\label{Propo_b}
Denote $(\hat{A}_b, \hat{B}_b, \hat{C}_b)$ as the modal ROM constructed using Algorithm \ref{BPOD_modal}, $(\Phi_b, \Psi_b)$ are BPOD modal bases.  Then $\hat{A}_b = \Phi_b A \Psi_b = \Lambda_{co}, \hat{B}_b = \Phi_b' B = U_{co}' B, \hat{C}_b = C \Psi_b = CV_{co}$, where $(\Lambda_{co}, U_{co}, V_{co})$ are the controllable and observable modes of the system, and $\hat{C}_b \hat{A}_b^i \hat{B}_b = CA^i B, i = 1, 2, \cdots$.
\end{proposition}

\begin{proof}
Consider the snapshots in the primal snapshot ensemble (\ref{sp}), 
\begin{eqnarray}\label{bpod_1}
x_i(t_k) =  A^{t_k} b_i,
\end{eqnarray}
where $i = 1, \cdots, p$ and $k = 1, \cdots, m$. Hence,
\begin{eqnarray}\label{solution}
\begin{pmatrix} x_1(t_k), & \cdots,  & x_p(t_k) \end{pmatrix} = A^{t_k} B, 
\end{eqnarray}
 and the snapshot ensemble $X_b$ can be written as:
\begin{eqnarray}\label{snap_ab}
X_b = \begin{pmatrix} A^{t_1 }B & A^{t_2 }B & \cdots & A^{t_m }B \end{pmatrix}.
\end{eqnarray}

Under assumptions A1 and A3,  and from (\ref{partition}), we have:
\begin{eqnarray}\label{snap_time}
A^{t_k} B = \begin{pmatrix} V_{co} & V_{c \bar{o}} \end{pmatrix} \begin{pmatrix} \Lambda_{co} & \\ & \Lambda_{c \bar{o}} \end{pmatrix}^{t_k} \begin{pmatrix} U_{co}' \\ U_{c \bar{o}}' \end{pmatrix} B, \nonumber \\
= \begin{pmatrix} V_{co} & V_{c \bar{o}} \end{pmatrix} \begin{pmatrix} \alpha_{co}^k \\ \alpha_{c \bar{o}}^k \end{pmatrix},
\end{eqnarray}
where $\alpha_{co}^k, \alpha_{\bar{c} o}^k$ are the coefficient matrices. Substitute (\ref{snap_time}) into (\ref{snap_ab}), and  $X_b$ can be written as:
\begin{eqnarray}\label{RPOD_snapshot}
{X}_b  = \begin{pmatrix} V_{co} & V_{c \bar{o}} \end{pmatrix} \begin{pmatrix} {\alpha}_{co}^b \\ {\alpha}_{c \bar{o}}^b \end{pmatrix},
\end{eqnarray}
where  ${\alpha}_{co}^b,  {\alpha}_{c \bar{o}}^b$ the coefficient matrices. Similarly,
\begin{eqnarray}\label{RPOD_adsnapshot}
{Z}_b = \begin{pmatrix} U_{co} & U_{\bar{c}o} \end{pmatrix} \begin{pmatrix} {\beta}_{co}^b \\ {\beta}_{\bar{c}o}^b \end{pmatrix}, 
\end{eqnarray}
where ${\beta}_{co}^b, {\beta}_{\bar{c}o}^b$ are some coefficient matrices. 

Hence, 
\begin{eqnarray}
Z_b' X_b = ((\beta_{co}^b)' U_{co}' + (\beta_{\bar{c}o}^b)' U_{\bar{c}o}')(V_{co} \alpha_{co}^b + V_{c \bar{o}} \alpha_{c \bar{o}}^b), \nonumber \\
=  (\beta_{co}^b)'\alpha_{co}^b, 
\end{eqnarray}
where $U_{co}' V_{c \bar{o}} = 0, U_{\bar{c}o}' V_{co} = 0, U_{\bar{c}o}' V_{c \bar{o}} = 0$.

Under assumption A3, there are exactly $l$ non-zero singular values in the resulting SVD problem, i.e.,
\begin{eqnarray}\label{start}
{Z}_b' {X}_b  = {L}_b {\Sigma}_b {R}_b', 
\end{eqnarray}
where ${\Sigma}_b \in \Re^{l \times l}$ are the $l$ non-zero singular values and $({L}_b, {R}_b)$ are the corresponding left and right singular vectors. Moreover,
\begin{eqnarray}
{Z}_b' A {X}_b = (\beta_{co}^b)' \Lambda_{co} \alpha_{co}^b. 
\end{eqnarray}

Now, consider the BPOD ROM:
\begin{eqnarray}\label{bpodor}
{A}_b = S_b A T_b = {\Sigma}_b^{-1/2} {L}_b' ({Z}_b' A {X}_b) {R}_b {\Sigma}_b^{-1/2} \nonumber \\
= \underbrace{{\Sigma}_b^{-1/2} {L}_b' (\beta_{co}^b)'}_{P_b} \Lambda_{co} \underbrace{\alpha_{co}^b {R}_b {\Sigma}_b^{-1/2}}_{\hat{P}_b}. 
\end{eqnarray}

We show that $\Lambda_{co}$ are the eigenvalues of ${A_b}$, and $P_b$ are the eigenvectors as follows:
\begin{eqnarray}
P_b \hat{P}_b = { \Sigma}_b^{-1/2} {L}_b' (\beta_{co}^b)' \alpha_{co}^b {R}_b {\Sigma}_b^{-1/2} \nonumber \\
= {\Sigma}_b^{-1/2} {L}_b' {L}_b {\Sigma}_b {R}_b' {R}_b {\Sigma}_b^{-1/2} = I.
\end{eqnarray}

Also,
\begin{eqnarray}\label{bpodeig}
\hat{P}_b P_b = \alpha_{co}^b {R}_b  {\Sigma}_b^{-1/2} {\Sigma}_b^{-1/2} {L}_b' (\beta_{co}^b)' \nonumber \\
= \alpha_{co}^b ((\beta_{co}^b)' \alpha_{co}^b)^{+} (\beta_{co}^b)' = I,
\end{eqnarray}
where $(.)^{+}$ denotes the pseudoinverse of $(.)$. Hence, $\hat{P}_b = P_b^{-1}$ and from (\ref{bpodor}),
\begin{eqnarray}
\Lambda_{co} = \underbrace{(P_b^{-1} S_b)}_{\Phi_b} A \underbrace{(T_b P_b)}_{\Psi_b}.
\end{eqnarray}

Also, we prove that $\Psi_b, \Phi_b$ are biorthogonal as follows.
\begin{eqnarray}
\Psi_b \Phi_b = T_b P_b P_b^{-1} S_b = T_b S_b, \nonumber \\
\Phi_b \Psi_b = P_b^{-1}S_b T_b P_b = P_b^{-1} P_b = I, 
\end{eqnarray}
where $(T_b, S_b)$ are BPOD bases, and are biorthogonal. 

Hence,
\begin{eqnarray}
{\Psi}_b = {T}_b P_b = {X}_b {R}_b {\Sigma}_b^{-1/2}  {\Sigma}_b^{-1/2} {L}_b' (\beta_{co}^b)' \nonumber \\
= (V_{co} \alpha_{co}^b + V_{c \bar{o}} \alpha_{c \bar{o}}^b)  {R}_b {\Sigma}_b^{-1} L_b' (\beta_{co}^b)'  \nonumber \\ = V_{co}\hat{P}_b P_b + V_{c \bar{o}} C_6 = V_{co} + V_{c \bar{o}} C_6,
\end{eqnarray}
where $C_6 = \alpha_{c \bar{o}}^b R_b \Sigma_b^{-1} L_b' (\beta_{co}^b)'$.  Similarly, 
\begin{eqnarray}
{\Phi}_b = P_b^{-1} S_b = U_{co}' + C_7U_{\bar{c}o}',
\end{eqnarray}
where $C_7 = \alpha_{co}^b R_b \Sigma_b^{-1} L_b' (\beta_{\bar{c}o}^b)'$. Under assumption A3,  the modal BPOD ROM constructed using $({\Psi}_b, {\Phi}_b)$ is:
\begin{eqnarray} \label{diagonalized_bpod}
\hat{A}_b = \Phi_b A \Psi_b =  \Lambda_{co}, \nonumber \\
\hat {B}_b= \Phi_b' B = U_{co}' B + C_7 U_{\bar{c}o}' B = U_{co}' B, \nonumber \\
 \hat{C}_b = C \Psi_b = CV_{co} + C V_{c \bar{o}} C_6 = C V_{co}.
\end{eqnarray}
And the Markov parameters of the ROM are:
\begin{eqnarray}
\hat{C}_b \hat{A}_b^i \hat{B}_b 
= C V_{co} \Lambda_{co}^i U_{co}' B = CA^iB.
\end{eqnarray}
\end{proof}


\textbf{Discussion on Proposition \ref{Propo_b}.} Recall the impulse response snapshot ensembles collected in the BPOD are:
\begin{eqnarray}
{X}_b = \underbrace{V_{co}}_{N \times l} \underbrace{\alpha_{co}^b}_{l \times pm} + \underbrace{V_{c \bar{o}} \alpha_{c \bar{o}}^b}_{N \times pm}, {Z}_b = \underbrace{U_{co}}_{N \times l} \underbrace{\beta_{co}^b}_{l \times qn} + \underbrace{U_{\bar{c}o} \beta_{\bar{c}o}^b}_{N \times qn},
\end{eqnarray}
and the Hankel matrix is:
\begin{eqnarray}
{H}_b = {Z}_b' {X}_b  = \underbrace{(\beta_{co}^b)'}_{qn \times l} \underbrace{\alpha_{co}^b}_{ l \times pm} \in \Re^{qn \times pm}.
\end{eqnarray}

There are two main parts to the computation:
\begin{enumerate}
\item The primal and adjoint snapshot ensembles ${X}_b \in \Re^{N \times pm}, {Z}_b \in \Re^{N \times qn}$, and hence, the construction of ${H}_b$ takes time $O(pqmnN)$.
\item The computational cost to solve the SVD of ${H}_b$ is $O(\min \{ p^2m^2qn, pmq^2n^2 \})$.
\end{enumerate}

However, $H_b$ is only rank ``$l$", where $l \ll N, pm, qn$, and from the development of Proposition \ref{Propo_b}, we see that the modal BPOD ROM given in (\ref{diagonalized_bpod}) is completely determined by
the $l$ controllable and observable modes and is invariant to the data $X_b$ and $Z_b$, i.e., as long as the snapshot ensembles can be written as:
\begin{eqnarray}\label{snap_basis}
\underbrace{X^*}_{N \times m} = \underbrace{V_{co} }_{N \times l} \underbrace{\alpha^*}_{l \times m} + \underbrace{V_{c \bar{o}} \bar{\alpha}^*}_{N \times m},  
\end{eqnarray}
\begin{eqnarray}
\underbrace{Z^*}_{N \times n} = \underbrace{U_{co} }_{N \times l} \underbrace{\beta^*}_{l \times n} + \underbrace{U_{\bar{c}o} \bar{\beta}^*}_{N \times n},
\end{eqnarray}
where $\alpha^*, \beta^*$ are rank $l$ constant matrices, and $\bar{\alpha}^*, \bar{\beta}^*$ are some constant matrices of suitable dimensions, then under assumptions A1, A2, and A3, the following corollary holds. 

\begin{corollary}\label{c1}
Denote $(A^*, B^*, C^*)$ as the modal ROM constructed using Algorithm \ref{BPOD_modal} with snapshot ensembles $X^*, Z^*$ as in (\ref{snap_basis}). Then $A^* = \Lambda_{co}, B^* = U_{co}' B, C^* = C V_{co}$, where $(\Lambda_{co}, U_{co}, V_{co})$ are the controllable and observable modes of the system, and $C^* (A^*)^i B^* = CA^iB, i = 1, 2, \cdots$.
\end{corollary}

Corollary \ref{c1} can be proved by replacing $X_b, Z_b$ in Proposition \ref{Propo_b} with $X^*, Z^*$, and the proof is omitted here. 
We make the following observations. 

\textbf{Observation:} 1) The snapshot ensembles do not have to be collections of the impulse responses as in BPOD.
2) Only $l$ snapshots may be enough  to extract all the controllable and observable modes of the system.

Bearing this observation in mind,  in the next section,  we introduce the RPOD$^*$ algorithm which generates the snapshot ensembles using exactly ``$l$" snapshots, such that the Corollary \ref{c1} holds.
\section{Computationally Optimal Randomized Proper Orthogonal Decomposition (RPOD$^*$)}\label{RPOD_body}
In this section, first, we define the computationally optimal snapshot ensemble, and prove that the snapshot ensembles generated by perturbing the primal/adjoint system with white noise are computationally optimal snapshot ensembles. Then we propose the RPOD$^*$ algorithm in Section \ref{br}. We discuss implementation issues of the algorithm in Section \ref{RPOD_issue}.
\subsection{Computationally Optimal Snapshot Ensemble}\label{opt_snapshot}
Under assumptions A1, A2 and A3, we define a computationally optimal snapshot ensemble of the system as follows.

\textbf{Definition 3}: A computationally optimal primal snapshot ensemble $X^*$ is an $l$-snapshot ensemble of rank $l$ which can be written as in Eq. (\ref{snap_basis}), i.e., $ m = l$.

A similar definition suffices for the computationally optimal adjoint snapshot ensemble $Z^*$.

Consider the system (\ref{primal})-(\ref{adjoint}), under assumption A1 that $A$ is stable, there exists a finite number $t_{ss}$, such that $\| A^{t_{ss}}\| \approx 0$. Under assumptions A1, A2 and A3, we have the following result.
\begin{proposition}\label{Po}
Perturb  the primal system (\ref{primal}) with white noise $u_k$, and collect $m$ snapshots at time $t_1, t_2, \cdots, t_m$, where $t_m \geq t_{ss}$, and $\|A^{t_{ss}} \| \approx 0$. Denote the snapshot ensemble as $X_r = \begin{pmatrix} x_1 & x_2 & \cdots & x_m \end{pmatrix}$. If $m = l$, where $l$ is the number of the controllable and observable modes of the system, then $X_r$ is a computationally optimal snapshot ensemble.
\end{proposition}
\begin{proof}
For the snapshots taken before $t_{ss}$, the state snapshot $x_k$ at time $k$ is:
\begin{eqnarray}\label{before}
x_k = \sum_{i = 0}^{k -1} A^{i} B u(k - i), k \leq t_{ss}.
\end{eqnarray}

Suppose there is a snapshot ensemble $X_f$ which takes $t_{ss}$ snapshots at time $k  = 1 $ to $k = t_{ss}$, then
from (\ref{before}), the snapshot ensemble $X_f$ can be written as:
\begin{eqnarray}\label{upper_triangular}
X_f = \begin{pmatrix} x_1 & x_2 & \cdots & x_{t_{ss}} \end{pmatrix} 
=\underbrace{\begin{pmatrix} B & AB & \cdots & A^{t_{ss} - 1}B \end{pmatrix}}_{X_b} \times \nonumber \\
 \underbrace{\begin{pmatrix} u(1) & u(2) & \cdots & u(t_{ss} - 1)& u(t_{ss}) \\ 0 & u(1) & \cdots&u(t_{ss} -2) & u(t_{ss} - 1) \\ 0 & 0 & \cdots &u(t_{ss} - 3)& u(t_{ss} -2) \\ \vdots & \vdots & \cdots & \cdots & \vdots\\ 0 & 0 & \cdots & 0& u(1) \end{pmatrix}}_{\Omega},
\end{eqnarray}
where $X_b$ is the BPOD snapshot ensemble from time $k = 0 $ to $k = t_{ss} -1$. Denote $\Omega = \begin{pmatrix} \omega_1 & \omega_2 \cdots & \omega_{t_{ss}} \end{pmatrix}$, where $\omega_i$ is the $i^{th}$ column of $\Omega$.
The $\Omega$ matrix above has columns that are linearly independent since it is upper triangular.  

Since $X_r$ consists of $m$ columns of $X_f$, $X_r$ can be written as:
\begin{eqnarray}\label{Fac}
X_r = \begin{pmatrix} x_1 & \cdots & x_m \end{pmatrix} = X_b \underbrace{\begin{pmatrix} \omega_1 & \cdots & \omega_m \end{pmatrix}}_{\Omega_1},
\end{eqnarray}
where $\begin{pmatrix} \omega_1 & \cdots & \omega_m \end{pmatrix}$ are the corresponding columns of $\Omega$, and hence, $\Omega_1$ has full column rank.

For the snapshot $x_k$ taken after $t_{ss}$, $x_k$ could also be written as: 
$x_k = X_b \omega_k,$
where $\omega_k$ is a column vector whose entries are white noises. Therefore, $\omega_k$ is independent of $\omega_1, \cdots, \omega_m$ in (\ref{Fac}), and hence, for all the snapshots collected in $X_r$, $X_r = X_b \Omega_1$, where $\Omega_1$ has full column rank. 




Recall that from (\ref{RPOD_snapshot}), $X_b$ can be written as:
\begin{eqnarray}
X_b = \underbrace{V_{co} }_{N \times l} \underbrace{\alpha_{co}^b}_{l \times pt_{ss}} + \underbrace{V_{c \bar{o}} \alpha_{c \bar{o}}^b}_{N \times pt_{ss}}, 
\end{eqnarray}
and $\alpha_{co}^b$ has full row rank. Thus, when $m = l$, 
\begin{eqnarray}
X_r = X_b \Omega_1 =\underbrace{V_{co}}_{N \times l} \underbrace{\alpha_{co}^b}_{l \times pt_{ss}} \underbrace{\Omega_1}_{pt_{ss} \times l} + V_{c \bar{o}} \alpha_{c \bar{o}}^b \Omega_1, \nonumber \\
= \underbrace{V_{co}}_{N \times l} \underbrace{\alpha_{co}}_{l \times l} + V_{c \bar{o}} \alpha_{c \bar{o}}, 
\end{eqnarray}
where rank $\alpha_{co} = l$. Hence, $X_r$ is a computationally optimal snapshot ensemble. 
\end{proof}

Similarly, we take $n = l$ snapshots by perturbing the adjoint system (\ref{adjoint}) with white noise $v_k$, and the adjoint snapshot ensemble can be written as:
\begin{eqnarray}
Z_r = \underbrace{U_{co}}_{N \times l}  \underbrace{\beta_{co}}_{l \times l} + U_{\bar{c} o} \beta_{\bar{c} o}, 
\end{eqnarray}
where rank $\beta_{co} = l$, and $\beta_{\bar{c}o}$ is a constant matrix. Thus, $Z_r$ is a computationally optimal snapshot ensemble. In Section \ref{RPOD_issue},  we discuss about how to choose $m, n$ and the snapshots in practice. 
\subsection{RPOD$^*$ Algorithm}\label{br}
The RPOD$^*$ algorithm is summarized in Algorithm \ref{RPOD_algo}. 

\begin{algorithm}[bt]
\begin{enumerate}
\item{ Perturb the primal system (\ref{primal}) with white noise $u_k$, collect $m$ snapshots at time step $t_1, t_2, \cdots, t_m$, where $t_m \geq t_{ss} $, $\| A^{t_{ss}} \| \approx 0$, $m \geq l$. Denote the snapshot ensemble $X_r$ as:
\begin{eqnarray}\label{algo_pri}
X_r = \begin{pmatrix} x_1 & x_2 & \cdots & x_m \end{pmatrix} .
\end{eqnarray}
 }
\item{Perturb the adjoint system (\ref{adjoint})  with white noise $v_k$, collect $n$ snapshots at time step $\hat{t}_1, \hat{t}_2, \cdots, \hat{t}_n$, where $\hat{t}_n \geq t_{ss} $, $n \geq l$. Denote the adjoint snapshot ensemble $Z_r$ as:
\begin{eqnarray}\label{algo_adj}
Z_r = \begin{pmatrix} z_1 & z_2 & \cdots & z_n \end{pmatrix}.
\end{eqnarray}
}
\item{Solve the SVD problem: 
\begin{eqnarray}\label{RPOD_SVD}
Z_r' X_r =\begin{pmatrix} L_r & L_o \end{pmatrix} \begin{pmatrix} \Sigma_r & 0 \\ 0 & \Sigma_o \end{pmatrix} \begin{pmatrix} R_r' \\ R_o' \end{pmatrix}, 
\end{eqnarray}
and truncate at $\sigma_l$, where $l$ is the number of controllable and observable modes present in the snapshot ensembles.  $\Sigma_r$ contains  the first $l$ non-zero singular values $\sigma_{1}  \geq  \sigma_2  \geq \cdots, \geq \sigma_{l} > 0$, $(R_r, L_r)$ are the corresponding right and left singular vectors. }
\item {Construct the POD bases: 
\begin{eqnarray}
T_r  = X_r R_r \Sigma_r^{-1/2}, S_r = \Sigma_r^{-1/2} L_r' Z_r'.
\end{eqnarray}}
\item{Construct the ROM $\tilde{A}$, find the eigenvalues $\Lambda_{r}$  and eigenvectors $P_{r}$ of $\tilde{A}$. 
\begin{eqnarray}
\tilde{A} = S_r A T_r = P_r \Lambda_{r} P_{r}^{-1},
\end{eqnarray}}
\item{Construct new POD bases: 
\begin{eqnarray}
\Phi_{r} = P_{r}^{-1} S_r, \Psi_{r} = T_r P_{r}.
\end{eqnarray}}
\item{The ROM is: 
\begin{eqnarray}
A_r = \Phi_{r} A \Psi_{r}, B_r = \Phi_{r} B, C_r = C \Psi_{r}
\end{eqnarray}}
\end{enumerate}
\caption{RPOD$^*$ Algorithm}\label{RPOD_algo}
\end{algorithm}

Under assumptions A1, A2 and A4, the following result holds.

\begin{proposition}\label{P1}
Denote $(A_r, B_r, C_r)$ as the ROM constructed using RPOD$^*$ following Algorithm \ref{RPOD_algo}. If we keep the first $l$ non-zero singular values in (\ref{RPOD_SVD}),  then $\| C_r A_r^i B_r - C A^i B \| \propto O(\epsilon),  i = 1, \cdots $, where $\epsilon$ is a small number defined in assumption A4.
\end{proposition}

The proof of Proposition \ref{P1} uses perturbation theory \cite{perturbation1,perturbation2} to extend  the proof of the idealized Proposition \ref{Propo_b} such that A4 holds instead of A3. Under assumption A4, the actual snapshot ensembles can be written as:
\begin{eqnarray}\label{snap_p1}
X_r = \begin{pmatrix} V_{co} & V_{c \bar{o}} & V_{\bar{c} o} & V_{\bar{c} \bar{o}} \end{pmatrix} \begin{pmatrix} \alpha_{co} \\ \alpha_{c \bar{o}} \\ \delta \alpha_{\bar{c}o} \\ \delta \alpha_{\bar{c} \bar{o}} \end{pmatrix},
\end{eqnarray}
where $\delta \alpha_{\bar{c} o}  = \epsilon \alpha_{\bar{c} o}$, $\delta \alpha_{\bar{c} \bar{o}} = \epsilon \alpha_{\bar{c} \bar{o}}$ and $\epsilon$ is small. 
Therefore, $\| V_{\bar{c} o } \delta \alpha_{\bar{c}o} + V_{\bar{c} \bar{o}} \delta \alpha_{\bar{c} \bar{o}} \| = O(\epsilon)$ are small perturbations of the ideal snapshot ensemble, and we can expect the ideal result to be perturbed by a small amount as well. 

The formal proof is shown in Appendex \ref{AP1}.
\begin{corollary}\label{cr2}
$\epsilon$ is assumed to be a small number in assumption A4, and $\epsilon$ can also be related to $\sigma_{l+1}$ as follows. 
\begin{eqnarray}
\| C_r A_r^i B_r - C A^i B \| \propto O(\epsilon) \propto O(\sigma_{l+1}). 
\end{eqnarray}
\end{corollary}
The proof is shown in Appendex \ref{AP2}.

\subsection{Implementation Issues}\label{RPOD_issue}
Here we discuss some  implementation problems in the RPOD$^*$ algorithm.  We give the insight into how to collect the snapshot ensembles, and how to select the ROM size. 

\textbf{Snapshot selection}
From the analysis in Section \ref{opt_snapshot},  we only need to collect $m = l$ snapshots from the primal simulations. However, $l$ is not known a priori, thus, in practice, we start with a random guess $m <<N$, where $N$ is the dimension of the system, or we can choose $m$ from experience. For instance, in a fluid system with $10^6$ degrees of freedom, $m$ is $O(10) \sim  O(10^2)$. Similarly, we guess $n$, and then we check the rank of $Z_r'X_r$. If $Z_r'X_r$ has full rank, i.e., rank $(Z_r'X_r) = \min{(m, n)}$,  then it is possible that we did not take enough snapshots, and hence, we increase $m, n$, until rank $(Z_r'X_r) < \min{(m,n)}$. 

For the primal simulation, we take $m$ snapshots from one simulation with zero initial condition, and white noise perturbation $u(k)$. We assume that the snapshots are taken at $\Delta T, 2 \Delta T \cdots, m \Delta T$, WLOG. Here, $\Delta T $ is a small constant, and we require that $m \Delta T \geq t_{ss}$, where $ \| A^{t_{ss}} \| \approx 0$. In Fig. \ref{snapshot}, we show one simulation result comparing the accuracy of the ROMs using $\Delta T = 3, 5, 10, 20, 50$ for the atmospheric dispersion problem introduced in Section \ref{atmos_pro}.  Here $t_{ss} = 900$, and $m = 300$. The output relative error is defined in (\ref{def_err}), and we take the average of the output relative error for each ROM. It can be seen that  as $\Delta T$ increases, each column in $\Omega_1$ is well separated, and hence, the ROM is more accurate, while it takes longer time to generate the snapshots. Thus, this is a trade-off between the accuracy and the computational efficiency.

\begin{figure}[bt]
\centering
\includegraphics[width=1.67in]{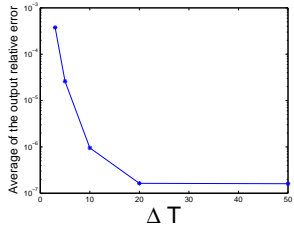}
\caption{Snapshot Selection: Effect of $\Delta T$}
\label{snapshot}
\end{figure}

\textbf{ROM size selection}
 In Proposition \ref{P1}, we assume that there are $l$ controllable and observable modes, and we keep exact $l$ non-zero singular values. However, $l$ is not known as a priori. We prove in \cite{RPODstar} that if  we keep $k$ non-zero singular values, and $k > l$, then undesired noise will be introduced. If $k < l$, not all the controllable and observable modes can be recovered. Therefore, in practice, we decide $l$ by trial and error. Since $\text{ rank } (Z_r'X_r) \geq l$ is always true, we start with $k = \text{ rank } (Z_r'X_r)$, and check the eigenvalues of $\tilde{A} = S A T$. From the development in \cite{RPODstar},  we can see that when $k > l$,  $ k - l$ eigenvalues of $\tilde{A}$ are small, which are the perturbations of the zero eigenvalues.  If $k > > l$, then the perturbations  in the ROM is too large to be neglected, which results in unstable eigenvalues of $\tilde{A}$. Thus, we keep decreasing the value of $k$ until $\tilde{A}$ is stable. If there are some small eigenvalues which are approximately  zero,  we know $k > l$, and we decrease the value of $k$ until we reach a region $[l ,  l+a]$, where $a$ is a small number, such that most of the eigenvalues of $\tilde{A}$ remain the same for different value of $k$ ($l$ controllable and observable modes with $k-l$ perturbations of the zero eigenvalues), then we stop and pick the number $l$ as the  number of non-zero eigenvalues of $\tilde{A}$. 

\section{Comparison with related Algorithms}\label{output}

In this section, we compare the RPOD$^*$ algorithm with BPOD, random projection and BPOD output projection algorithm. 
\subsection{Comparison with BPOD}\label{comp_bpod}
First, we summarize the differences between the BPOD and RPOD$^*$ algorithm. 

As we mentioned in Section \ref{BPOD_sec}, $p + q$ simulations are needed for BPOD algorithm, where $p$ is the number of inputs and $q$ is the number of outputs. Let $X_b$, $Z_b$ denote  the impulse response snapshot ensembles that need to be collected in BPOD algorithm from time step $(1, \cdots , t_{ss})$. Thus, $X_b \in \Re^{N \times pt_{ss}}, Z_b \in \Re^{N \times qt_{ss}}$. It is expensive to store $(p + q) t_{ss}$ snapshots, and it is expensive to solve the resulting SVD problem due to the large size of the problem. For RPOD$^*$ algorithm, only (1 primal + 1 adjoint) simulations are needed, and only $m + n$, where $m, n \ll t_{ss}$ snapshots need to be stored. Also, it is easy to solve the resulting SVD problem.  

Another practical problem with impulse responses snapshots is that the snapshots after some time are dominated by very few slow modes, and including these snapshots does not give much new information. On the other hand, the RPOD$^*$ trajectories are white noise forced, and all the modes are always present in all the snapshots due to the persistent excitation of the white noise terms. Hence, the RPOD$^*$ snapshots can be taken till $t_{ss}$, and be assured that all of the relevant modes will be captured.

\subsection{Comparison with Random Projection}\label{comp_rand}
From the analysis in Section \ref{opt_snapshot},  we see that the snapshot ensembles collected in RPOD$^*$ can be written as:
\begin{eqnarray}
X_r = X_b \Omega_1, Z_r = Z_b \Omega_2, 
\end{eqnarray}
where $X_b \in \Re^{N \times pt_{ss}}, Z_b \in \Re^{N \times qt_{ss}}$ are the impulse response snapshot ensembles that need to be collected in the BPOD algorithm from time step $(0, \cdots, t_{ss} -1)$, and $\| A^{t_{ss}} \| \approx 0$.  $\Omega_1 \in \Re^{pt_{ss} \times m}$ and $\Omega_2 \in \Re^{qt_{ss} \times n}$ are full rank matrices. We have:
\begin{eqnarray}
\underbrace{H_r}_{\text{rank } l} = Z_r' X_r = \Omega_2' {Z}_b' {X}_b \Omega_1 = \Omega_2' \underbrace{{H}_b}_{\text{rank } l} \Omega_1,
\end{eqnarray}

There is a significant difference  between the proposed algorithm and a direct application of the random projection algorithm on BPOD. A direct application of the random projection would require to generate the Hankel matrix ${H}_b$ (and ${X}_b, {Z}_b$) first. However, in practice, the construction and the storage of the Hankel matrix is computationally prohibitive when $N$ is large and the number of inputs/outputs is large. Also, the bottleneck of the random projection algorithm is the projection of ${X}_b, {Z}_b$ onto the Gaussian test matrices. In the proposed algorithm, the snapshot ensembles are constructed directly from the primal and adjoint simulations, and hence, the computational cost to generate the Hankel matrix and to project it onto the Gaussian test matrices is saved. 

\subsection{Comparison with BPOD output projection}\label{comp_bpodp}
As we mentioned in Section \ref{BPOD_sec}, when the number of the outputs $q$ is large, the computation of the BPOD adjoint simulations may not be tractable. To reduce the number of the outputs, the output projection method in \cite{rowley1, BPODprojection} is proposed. In this section, we compare the RPOD$^*$ algorithm with BPOD output projection algorithm. First, we briefly review the BPOD output projection method in Algorithm \ref{BPODp_algo}. 

\begin{algorithm}[bt]
\begin{enumerate}
\item Collect the primal snapshot ensemble $X_b$ in (5).
\item Compute the POD modes of the dataset $Y_b = CX_b$.
\begin{eqnarray}
Y_b'Y_b  \phi_k = \lambda_k \phi_k, \lambda_1 \geq \cdots \geq \lambda_n \geq 0, 
\end{eqnarray}
where $\lambda_k$ are the eigenvalues, and $\phi_k$ are the corresponding eigenvectors. Thus, the POD modes $\Theta_s  = [\phi_1, \cdots, \phi_s]$, where $s$ is the rank of the output projection, and $s \ll q$. 
\item Collect the impulse responses of the adjoint system:
\begin{eqnarray}\label{outputprojection}
z_{k+1} = A' z_k + C' \Theta_s v,  w = B' z_k.
\end{eqnarray}
The adjoint snapshot ensemble is denoted as $Z_s$.
\item Construct the Hankel matrix:
\begin{eqnarray}
H_s = Z_s' X_b.
\end{eqnarray}
\item Solve the SVD problem of $H_s$, and construct the BPOD output projection bases as $(T_s, S_s)$ using equations (8) and (9). The ROM is:
\begin{eqnarray}
A_s = S_s A T_s, B_s = S_s B, C_s = C T_s.
\end{eqnarray}
\end{enumerate}
\caption{BPOD output projection algorithm}\label{BPODp_algo}
\end{algorithm}
In the following, we compare the BPOD output projection method with RPOD$^*$ algorithm. From (\ref{outputprojection}), the adjoint snapshote ensemble $Z_s$ can be written as:
\begin{eqnarray}
\underbrace{Z_s}_{N \times st_{ss}} = \begin{pmatrix} C' \Theta_s & A' C' \Theta_s & \cdots & (A')^{t_{ss}} C' \Theta_s \end{pmatrix} \nonumber \\
= \underbrace{\begin{pmatrix} C' & A'C' & \cdots & (A')^{t_{ss}} C' \end{pmatrix}}_{Z_b} \underbrace{\begin{pmatrix} \Theta_s & & & \\ & \Theta_s & & \\ & & \cdots & \\ & & & \Theta_s \end{pmatrix}}_{\Theta} \nonumber \\
= \underbrace{Z_b}_{N \times qt_{ss}} \underbrace{\Theta}_{q t_{ss} \times s t_{ss}}.
\end{eqnarray}
And hence, the projected Hankel matrix $H_s$ is:
\begin{eqnarray}
H_s = \Theta'  Z_b' X_b = \Theta' H_b,
\end{eqnarray}
Recall that the adjoint snapshot ensembles collected in RPOD$^*$ can be written as $Z_r = Z_b \Omega_2$, and the projected Hankel matrix in RPOD$^*$ is $H_r= \Omega_2' H_b \Omega_1$. Thus, we make the following remark. 
\begin{remark}
Both the BPOD output projection and RPOD* algorithms can be viewed as projecting the full order Hankel matrix onto a reduced order Hankel matrix with projection matrices $\Theta$ and $(\Omega_1, \Omega_2)$. 
\end{remark}

\subsubsection{Differences between two algorithms}

First, the information preserved in $H_s, H_r$ are not the same. The output projection $P_s = \Theta_s \Theta_s'$ minimizes the 2-norm of the difference between the original transfer function and the output-projected transfer function, i.e. $\| C A^{i} B - \Theta_s \Theta_s' CA^{i} B \|, i = 1, \cdots, $ is minimized. Thus, the controllable and observable modes preserved in $H_s$ are an approximation of those in $H_b$, while in RPOD$^*$ algorithm, the exact  controllable and observable modes are preserved using the Gaussian random projection matrix. As mentioned in \cite{BPODprojection}, when $s < l$, where $l$ is the number of non-zero Hankel singular values (number of controllable and observable modes), then only the first $s$ Hankel singular values are the same as the full balanced truncations Hankel singular values. 

Another difference between output projection algorithm and RPOD$^*$ algorithm is that RPOD$^*$ algorithm can be used when both the number of the inputs and outputs are large, while output projection can be used when the number of the inputs or the outputs is large. When the number of inputs is large, we can construct an input projection using the adjoint snapshot ensemble, but when both the number of inputs and outputs are large, the construction of the projection matrix is not helpful. 

\subsubsection{Comparison of  computational cost}
The comparison of the  computational cost of the BPOD output projection algorithm with RPOD$^*$ algorithm is shown in Table \ref{compro}. We collect $m, n$ snapshots for RPOD$^*$ algorithm respectively,  and $t_{ss}$ snapshots for BPOD output projection algorithm. $N$ is the dimension of the system, $p, q$ are the number of inputs and outputs respectively, the rank of the output projection is $s$, and without loss of generality, we assume $p \leq q$, $m \leq n$, and $p \leq s$. 
\begin{table}
\caption{Computational Complexity Analysis for RPOD$^*$ and BPOD Output Projection}\label{compro}
\begin{tabular}{l|l|l}
\hline
& RPOD* & BPOD output projection \\
Construction of $H$ & $O(mnN)$ &  $O(pst_{ss}^2 N)$\\
\hline
Solve SVD & $O(m^2n)$ & $O(p^2 st_{ss}^3)$ \\
\hline
\end{tabular}
\end{table}


Now we compare the computation time to generate the snapshot ensembles. If we denote $T$ as the time to propagate the primal/adjoint system once, then for BPOD output projection algorithm, $(p + s)$ simulations are needed, and for each simulation, we need to collect the snapshots up to $t = t_{ss}$.  Thus, the total computation time to generate the snapshot ensembles is $(p+s) t_{ss}T$. 

The RPOD$^*$ algorithm needs 1 primal and 1 adjoint simulation till time $ m \Delta T, n \Delta T$ respectively, where $m \Delta T \geq t_{ss}, n \Delta T \geq t_{ss}$. When we choose $\Delta T$ such that $m \Delta T = t_{ss}, n \Delta T = t_{ss}$, the performance of the ROM constructed using RPOD$^*$ is better than those constructed using BPOD/BPOD output projection algorithm. While, the total computation time to generate the snapshot ensembles is $2t_{ss}T$, which means that RPOD$^*$ computational cost is  the same as BPOD in a single input single output (SISO) system. The comparison of the computation time to generate the snapshot ensembles is shown in Table \ref{time} for two examples.

\section{Computational Results}\label{example}
In this section, we show the comparison of RPOD$^*$ with BPOD for two examples: a one-dimensional heat transfer problem, and a three dimensional atmospheric dispersion problem to illustrate the proposed algorithm. 

First, we define the output relative error:
\begin{eqnarray}\label{def_err}
E_{output} = \frac{\| Y_{true} - Y_{rom} \|}{\| Y_{true} \| },
\end{eqnarray}
where $Y_{true}$ are the outputs of the full order system, and $Y_{rom}$ are the outputs of the reduced order system.

The frequency response for multiple input multiple output systems can be represented by plotting the maximum singular value of the transfer function matrix $\max(\sigma(H(j \omega)))$ as a function of frequency $\omega$. We define the frequency responses error as:
\begin{eqnarray}
E_{fre} (j \omega)= | \max(\sigma(H_{true} (j \omega))) - \max(\sigma( H_{rom}(j \omega)))|,
\end{eqnarray}
where $H_{true}(j \omega)$ is the transfer function of the full order system, and $H_{rom}(j \omega)$ is the transfer function of the ROM. 

For the two experiments below, there are design parameters that need to be chosen manually, such as the rank of the output projection, when to take the snapshots, and the size of the ROM. We  have shown the results of the best selections of these parameters for the BPOD/BPOD output projection algorithms and the RPOD$^*$ algorithm.
\subsection{ Heat Transfer}
The equation for heat transfer by conduction along a slab is given by the partial differential equation:

\begin{equation}
\frac{\partial T}{\partial t} = \alpha \frac{\partial^2 T}{\partial x^2}+f,
\end{equation}\\
 with boundary conditions
\begin{eqnarray}
T |_{x = 0} = 0,  \frac{\partial T}{\partial x} |_{x=L} = 0,
\end{eqnarray}
where $\alpha$ is the thermal diffusivity, and $f$ is the forcing.

Two point sources are located at $x = 0.15m $ and $x = 0.45m$. The system is discretized using finite difference method, and there are 100 grids which are equally spaced.  We take full field measurements, i.e., measurements at every node. The parameters of the system are summarized in Table \ref{parameters}.  In the following, we compare the ROM constructed using RPOD$^*$, BPOD and BPOD output projection.

For RPOD$^*$, the system is perturbed by the white noise with distribution $N(0, I_{2 \times 2})$. At time $t_{ss} = 3000s$, $ \| A^{t_{ss}}\| \approx 0$, thus, for the primal/adjoint simulation, one realization is needed, and we collect 80 equispaced snapshots during time $t \in [0, 3200]s$. For BPOD, we need $p = 2$ realizations for the primal simulation, and $q = 100$ realizations for the adjoint simulations. In general, 3000 equispaced snapshots between $[0, 3000]s$ should be taken in each realization for BPOD/output projection. However, due to the memory limits on the platform, only 400 equispaced snapshots can be taken and for the optimal performance of BPOD/output projection, the 400 equispaced snapshots are taken from time $t \in [0, 400]s$ (first 400 dominant impulse responses). The rank of the output projection is 40, and hence, for the BPOD output projection algorithm, 40 realizations are needed for the adjoint simulations. RPOD$^*$ algorithm and BPOD algorithm extract 70 modes, and BPOD output projection algorithm extract 50 modes. Here, we take 80 snapshots by trial and error following the procedure in Section \ref{RPOD_issue}. 
\begin{table*}[t]
\caption{Parameters of system}
\centering
 \begin{tabular}{c|c|p{4cm}|p{6cm}|}
\hline
             & & Heat & Atmospheric Dispersion\\
\hline

\multirow{5}{*}{System Parameters}
& Domain & $x \in [0, 1] (m)$ & $x \in [0, 2000] (m)$, $y \in [-100, 400] (m),$ $z \in  [0, 50] (m)$ \\
\cline{2-4}
& Dimension of system &  $N = 100$  &  $N = 10^5$ \\
\cline{2-4}
& Parameters &$\alpha = 4.2 \times 10^{-6} (m^2/s)$ & Wind velocity $\vec{u} = (4m/s, 0, 0 )$\\
\cline{2-4}
& Number of Inputs & 2 & 10\\
\cline{2-4}
& Number of Outputs & 100 & 810\\
\hline 
\multirow{5}{*}{output projection V.S. RPOD$^*$}
& Primal Snapshots & 400 V.S. 80 & 200 V.S. 400  \\
\cline{2-4}
& Adjoint Snapshots &  400 V.S. 80 & 200 V.S.  400\\
\cline{2-4}
& Snapshots taken during & $t \in [0, 400s]$ V.S. $t \in [0s,3200s]$ & $t \in [0, 200]s$ V.S. $ t \in [0, 4000]s$ \\
\cline{2-4}
& ROM modes& 50 V.S. 70 & 200 V.S. 380\\
\cline{2-4}
& Hankel matrix & $16000 \times 800$ V.S. $80 \times 80$ & $4000 \times 2000$ V.S. $ 400 \times 400$ \\
\hline
        \end{tabular}
\label{parameters}
\end{table*}

In Fig. \ref{heat_com}(a), we compare the norm of the Markov parameters of the ROM constructed using three algorithms with the true Markov parameters of the full order system.  We perturb the system with random Gaussian noise, and compare the output relative errors in Fig. \ref{heat_com}(b). In Fig. \ref{heat_frequency}(a), we compare the frequency responses of the ROM constructed using three algorithms with the true frequency responses. In Fig. \ref{heat_frequency}(b), we plot the error of the maximum singular value of the input-output model as a function of frequency. 

\begin{figure}[tb]
\centering
\subfigure[Comparison of Markov Parameters]{
\includegraphics[width=1.67in]{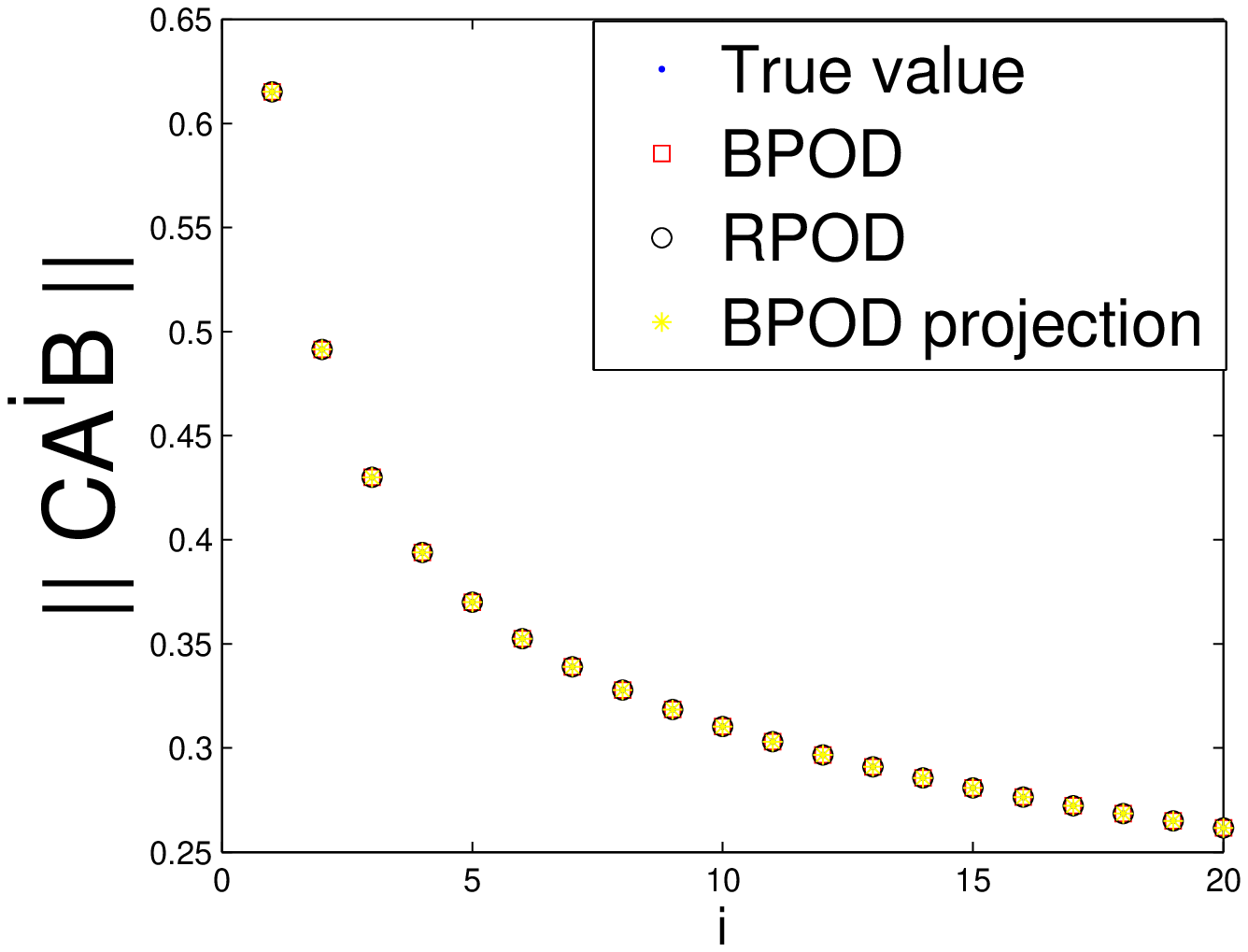}}
\subfigure[Comparison of  Output Relative Errors]{\includegraphics[width=1.67 in]{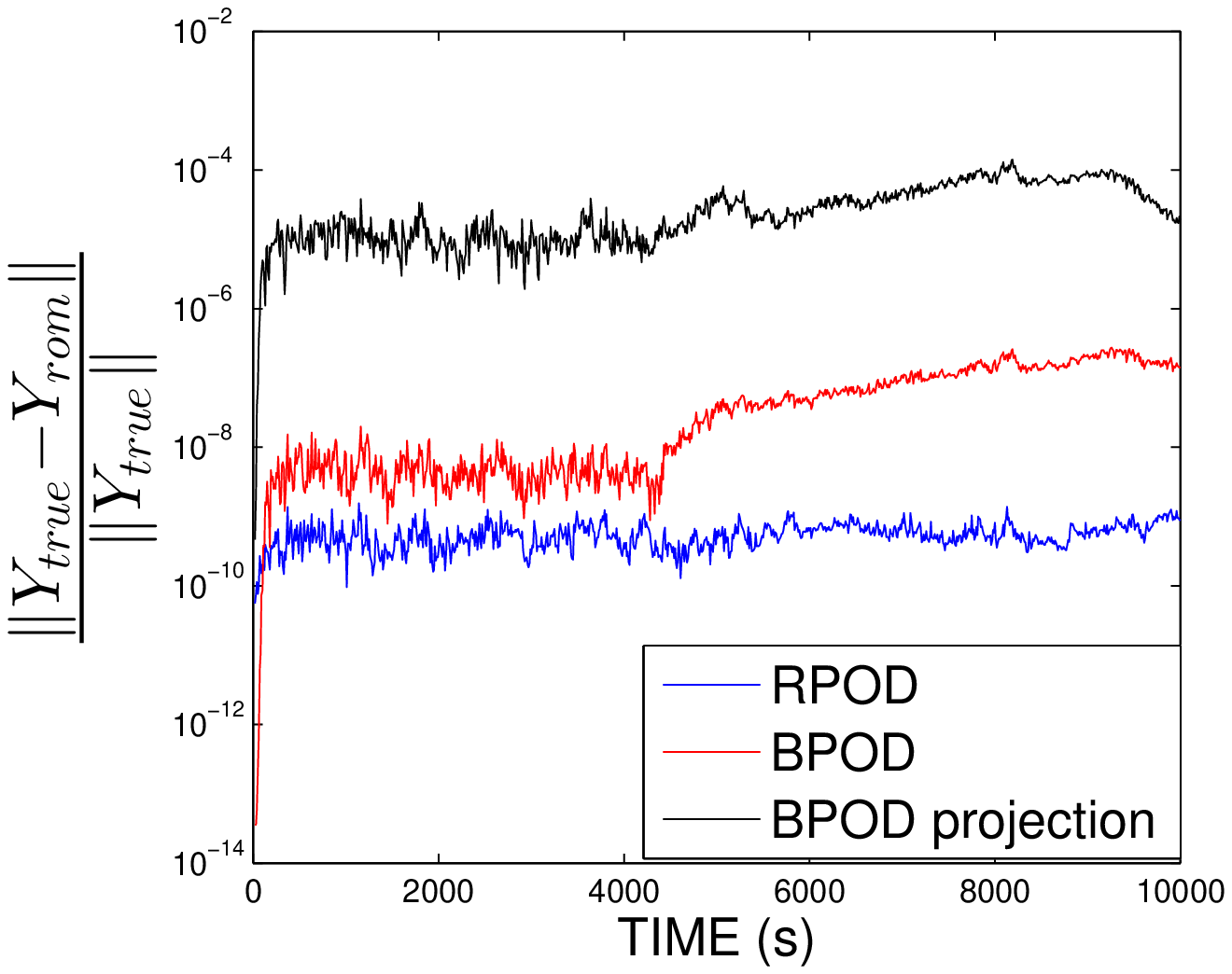}}
\caption{Heat Transfer Problem}
\label{heat_com}
\end{figure}

\begin{figure}[tb]
\centering
\subfigure[Comparison of Frequency Responses]{
\includegraphics[width = 1.67in]{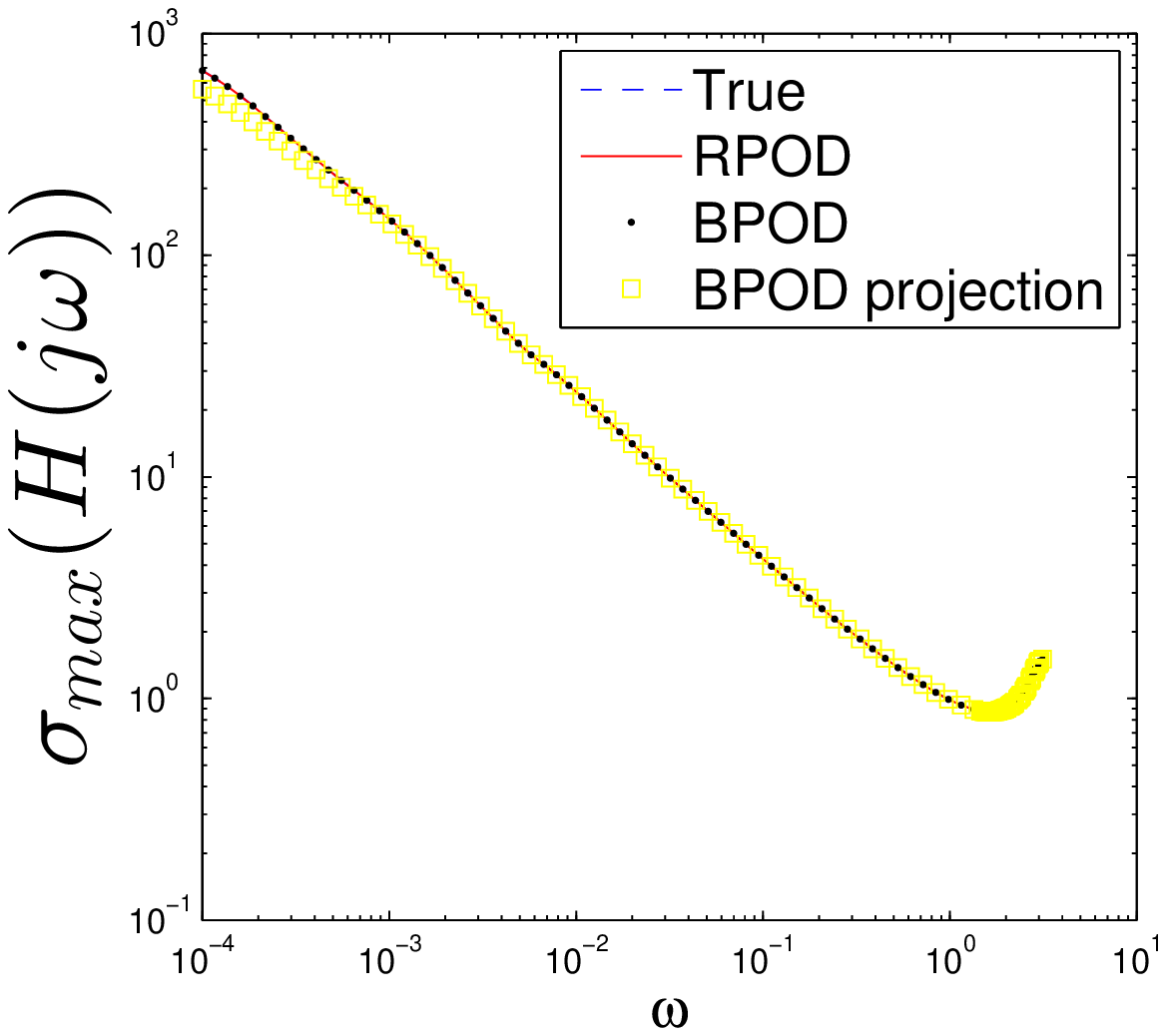}}
\subfigure[Comparison of  Frequency Responses Errors]{\includegraphics[width= 1.67in]{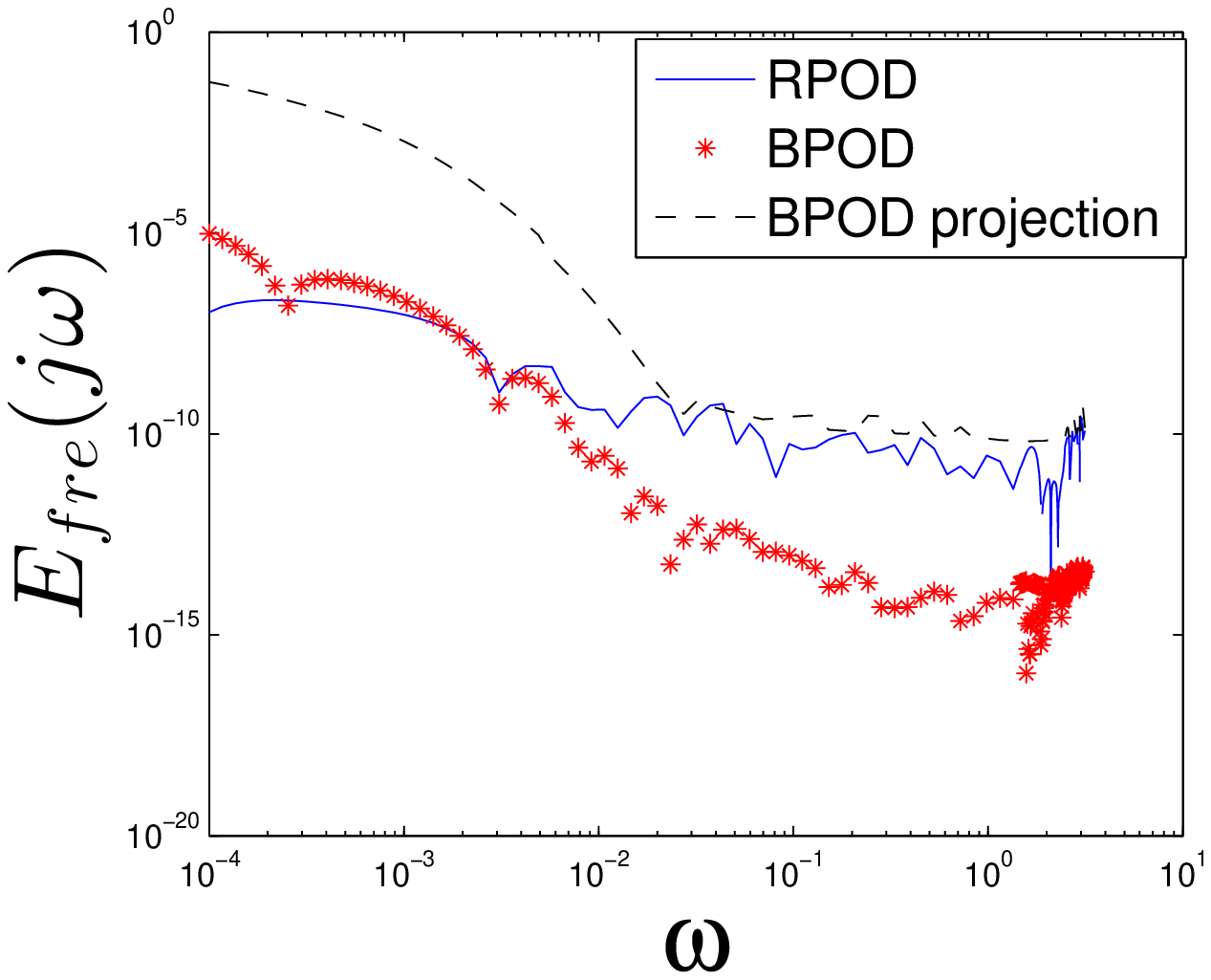}}
\caption{Heat Transfer Problem}
\label{heat_frequency}
\end{figure}

For all three methods, we can see that the Markov parameters of the ROM are close to the true Markov parameters of the full order system. From Fig. \ref{heat_com}(b), it can be seen that  all three methods are accurate enough. The output relative error of BPOD output projection  is less than $0.01 \%$, and the performance of the RPOD$^*$ algorithm is much better than the BPOD output projection algorithm, the performance of RPOD$^*$ algorithm is better than BPOD algorithm because we do not take all the snapshots up to $t_{ss}$ due to the memory limits. From Fig. \ref{heat_frequency}, we can see that the frequency responses error of RPOD$^*$ algorithm is  smaller than BPOD and BPOD output projection algorithm in low frequencies. With the increase of the frequency, the BPOD algorithm performs better than RPOD$^*$ algorithm, however, the errors are below $10^{-10}$, and hence, the difference is negligible. The comparison of computational time of RPOD$^*$ and BPOD output projection algorithm  is shown in Section \ref{com_time}. We can see that the construction of the snapshot ensembles using RPOD$^*$  takes almost the same time as BPOD output projection, while the dominant computational cost is solving the SVD problem, and it can be seen that RPOD$^*$ is about 24 times faster than BPOD output projection. 

\subsection{ Atmospheric Dispersion Problem}\label{atmos_pro}
The three-dimensional advection-diffusion equation describing the contaminant transport in the atmosphere is: 
\begin{eqnarray}
\frac{\partial c}{ \partial t} + \nabla \cdot (c \vec{u})  = \nabla \cdot (K(\vec{X}) \nabla c) + Q \delta( \vec{X} - \vec{X_s}),
\end{eqnarray}
where 

$c(\vec{X}, t)$ : mass concentration at location $\vec{X} = (x, y, z)$.

$\vec{X_s}  = (x_s, y_s, z_s)$:  location of the point source. 

$\vec{u}= (ucos(\alpha), usin(\alpha), 0)$:  wind velocity. $\alpha$ is the direction of the wind in the horizontal plane and the wind velocity is aligned with the positive $x$-axis when $\alpha = 0$, $u \geq 0$ is constant.

 $Q$:  contaminant emitted rate. 

$\nabla$:  gradient operator. 

$K(\vec{X}) = diag (K_x(x), K_y (x) , K_z (x))$ : diagonal matrix whose entries are the turbulent eddy diffusivities. In general $K(\vec{X})$ is a function of the downwind distance $x$ only.

In practice, the wind velocity is sufficiently large that the diffusion in the $x$-direction is much smaller than advection, and hence, assume that the term $K_x \partial_x^2 c$ can be neglected.  

Define $\sigma_{y}^2 (x) = \frac{2}{u} \int_{0}^x K_{y} (\eta) d \eta$, $\sigma_{z}^2 (x) = \frac{2}{u} \int_{0}^x K_{z} (\eta) d \eta $, where $\sigma_{y} (x) = a_{y} x (1 + b_{y} x)^{0.5}$, $\sigma_{z} (x) = a_{z} x (1 + b_{z} x)^{0.5}$, and $a_y = 0.008, b_y = 0.00001, a_z = 0.006,  b_z = 0.00015$.

The boundary conditions are:
\begin{eqnarray}
c(0, y, z) = 0, c(\infty, y, z) = 0, c(x, \pm \infty, z) = 0, \nonumber \\
c(x, y, \infty) = 0, K_z \frac{\partial c}{\partial z} (x, y, 0) = 0.
\end{eqnarray}

The system is discretized using finite difference method, and there are $100 \times 100 \times 10$ grids which are equally spaced. The parameters are summarized in Table \ref{parameters}. There are 10 point sources which are shown in Fig. \ref{atmo_com}. We take the full field measurements (except the nodes on $x = 0, \infty$ and $ y = \pm \infty$). In Fig. \ref{atmo_com}, we show the actual concentration of the full field at time $t = 200s$ with $Q$ as Gaussian white noise where sources are the dotted points in the figure.

In this example, since the system dimension is $N = 10^5$, constructing the ROM with the full field measurements using BPOD is computationally impossible, and thus, we only compare the RPOD$^*$ algorithm with BPOD output projection algorithm. 

For RPOD$^*$ algorithm, we collect the snapshots sequentially. The system is perturbed by white noise with distribution $N(0, I_{10 \times 10})$. One primal simulation and one adjoint simulation  are needed. We collect 400 equispaced snapshots  from time $t \in [0, 4000]s$, where at time $t_{ss} = 4000$,  $\|A^{t_{ss}}\| \approx 0$, and extract 380 modes. For BPOD output projection algorithm, we collect the impulse responses from the primal simulations, and $p = 10$ realizations are needed. Same as the heat example, we could not collect all the impulse responses from $t \in [0, 4000]s$ due to the memory limits  on the platform. Hence, for the best performance available in this example,  we collect 200 equispaced snapshots from $t \in [0, 200]s$ for each primal simulation. The rank of the output projection is 80, and hence, 80 adjoint simulations are needed, and in the adjoint simulations, we collect 50 equispaced snapshots from  $t \in [0, 50]s$. We can extract 200 modes. The parameters are summarized in Table \ref{parameters}.  

\begin{figure}[tb]
\centering
\includegraphics[width=1.67in]{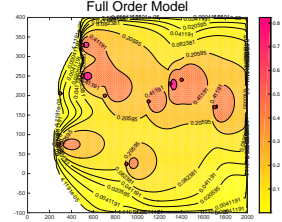}
\caption{Concentration Field at $t=200s$}
\label{atmo_com}
\end{figure}

In Fig. \ref{atmo_impulse}(a), we compare the Markov parameters of the ROM constructed using RPOD$^*$ and BPOD output projection with the full order system. Also, we perturb the system with random Gaussian noise, and compare the output relative errors in Fig. \ref{atmo_impulse}(b). The comparison of the computational time is shown in Section \ref{com_time}.

\begin{figure}[tb]
\centering
\subfigure[Comparison of Markov Parameters]{
\includegraphics[width=1.67in]{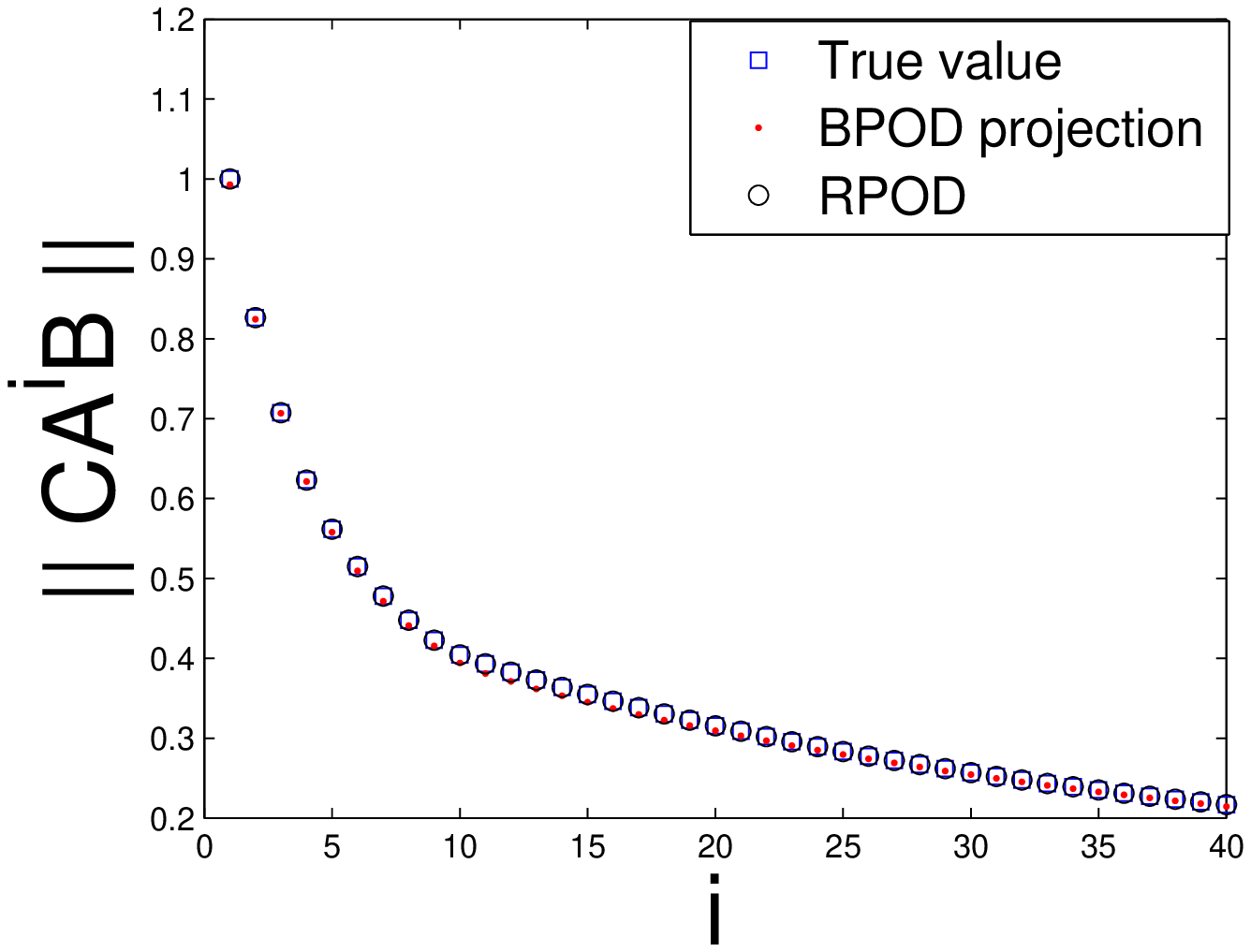}}
\subfigure[Comparison of  Output Relative Errors]{\includegraphics[width=1.67 in]{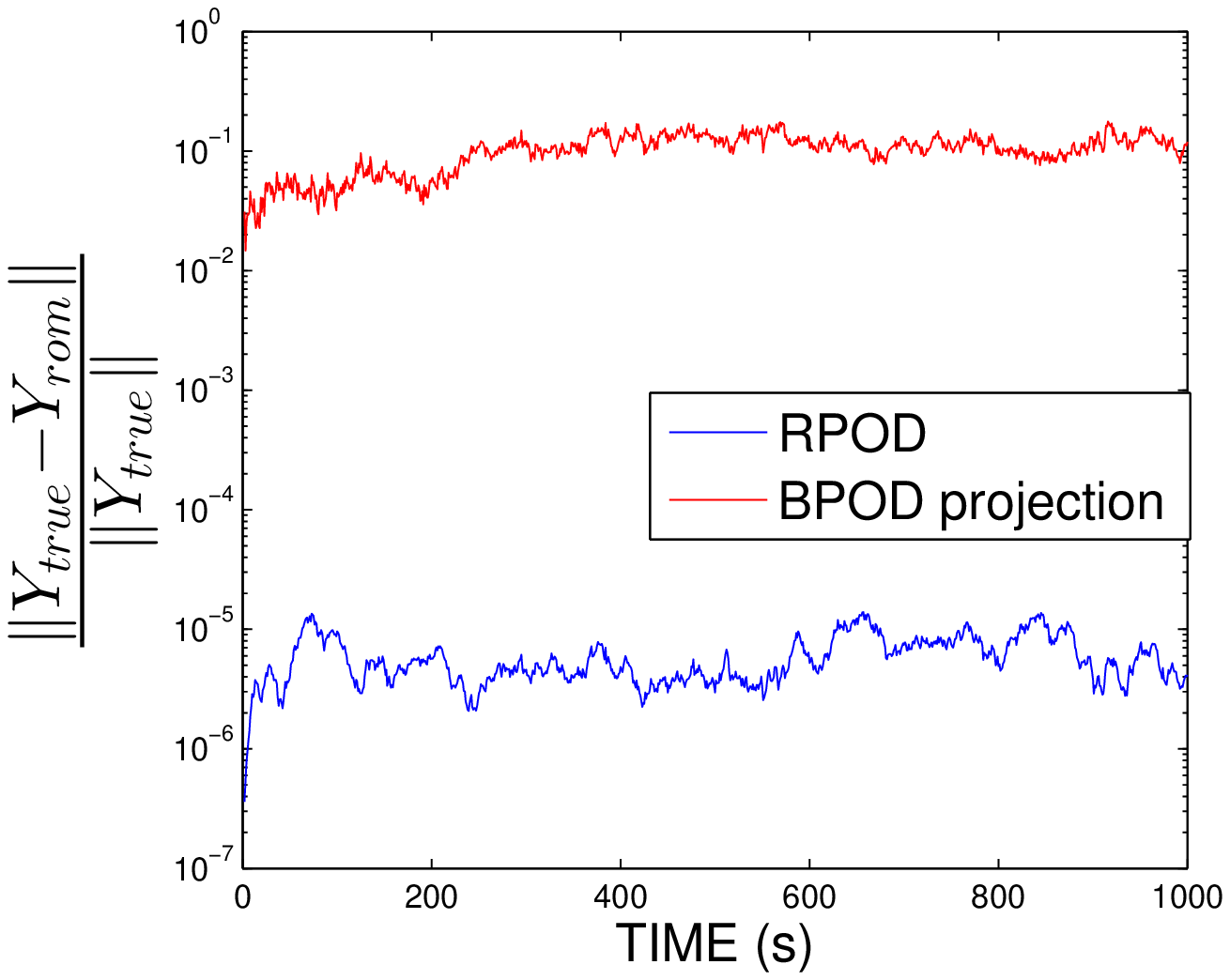}}
\caption{Atmospheric Dispersion Problem}
\label{atmo_impulse}
\end{figure}

In Fig. \ref{atmo_frequency}(a), we compare the frequency responses of  the ROM constructed using RPOD$^*$ and BPOD output projection with the full order system. We can see that  the frequency responses of the ROMs are almost the same as the frequency responses of the full order system. In Fig. \ref{atmo_frequency}(b), we show the errors between the frequency responses. 

\begin{figure}[tb]
\centering
\subfigure[Comparison of Frequency Responses]{
\includegraphics[width=1.67in]{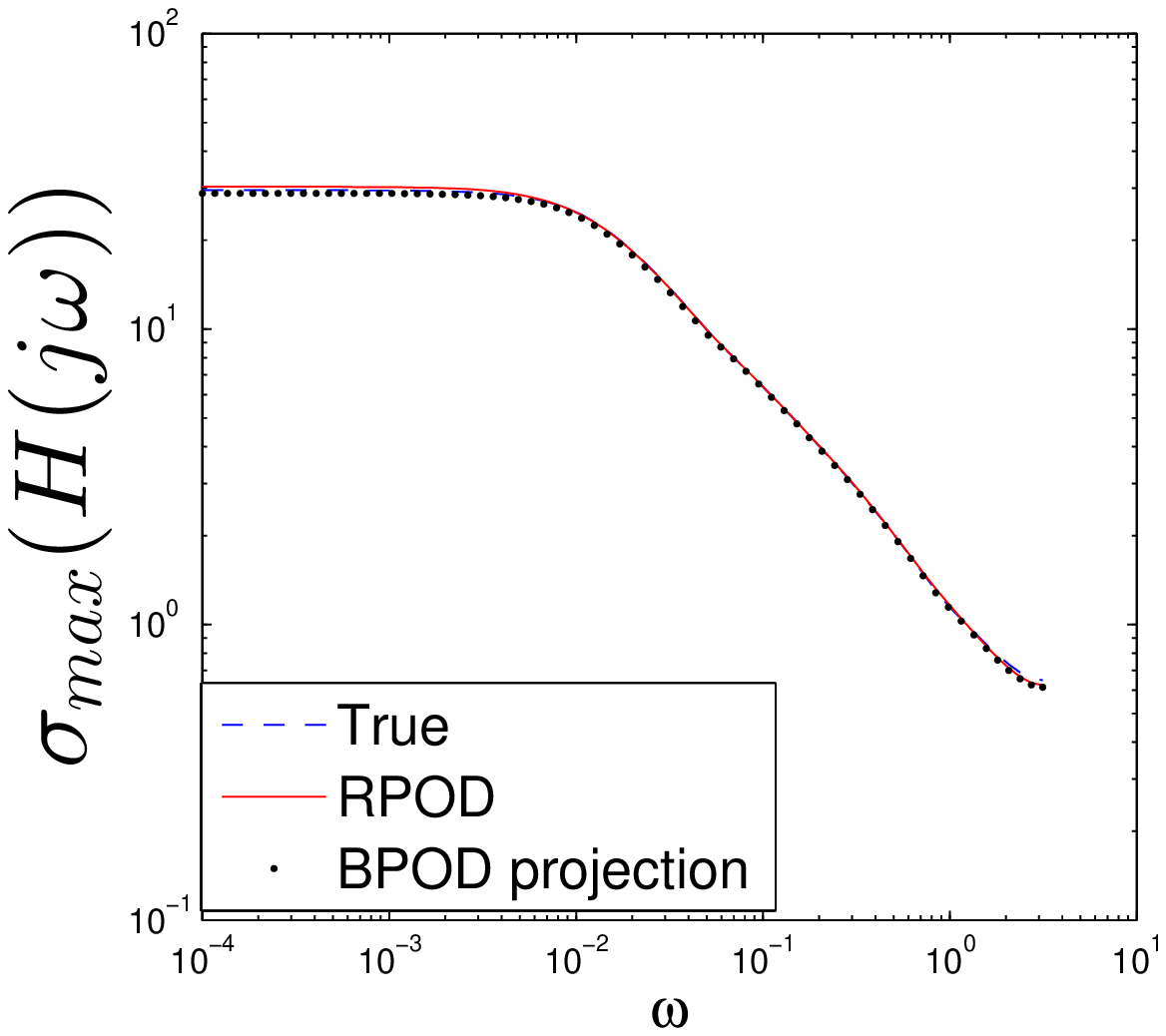}}
\subfigure[Comparison of Frequency Responses Errors]{\includegraphics[width=1.67 in]{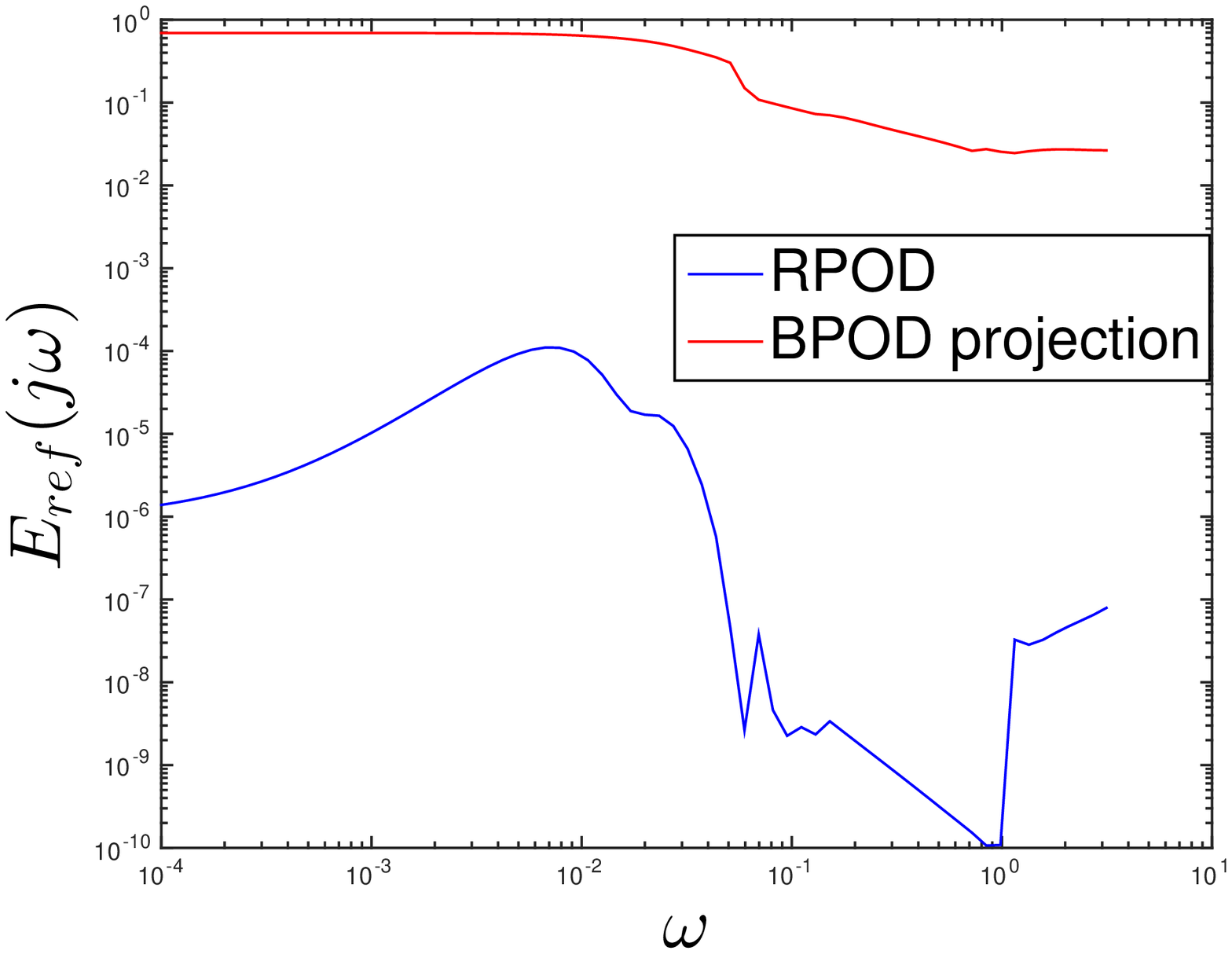}}
\caption{Atmospheric Dispersion Problem}
\label{atmo_frequency}
\end{figure}

The comparison of the computational time is shown in Section \ref{com_time}. It can be seen that for this example, the construction of the snapshot ensembles using RPOD$^*$ is faster than the  BPOD output projection, and the dominant computation cost is the construction of $Z'X$, where RPOD$^*$ algorithm is almost 16500 times faster than BPOD output projection.

From the examples above, we can see that for both examples showed in this paper, the RPOD$^*$ algorithm is much faster than BPOD/BPOD output projection algorithm, and is much more accurate than BPOD output projection algorithm. 

\subsection{Comparison of Computational Time} \label{com_time}
Comparison of computational time is shown in Table \ref{time} for two examples. All of the experiments reported in this paper were performed using  Matlab 2013b on a Dell OptiPlex 9020, Intel(R) Core (TM) i7-4770CPU, 3.40GHz, 4GB RAM machine.

\begin{table*}[bt]
\caption{Comparison of Computational Time}
\centering
 \begin{tabular}{c|c|c|c|c|}
\hline
\multicolumn{1}{c|}{} & \multicolumn{2}{|c|}{Heat} & \multicolumn{2}{|c|}{Atmospheric Dispersion}\\
\hline
& RPOD$^*$ & output projection & RPOD$^*$ & output projection\\
\hline
Generate $X$ & 0.0148s & 0.0518s & 55.59s & 30.461s\\
\hline
Generate $Z$ & 0.0340s & 0.0864s + 0.1582s(projection) &56.23s &421.99s + 9.287s(projection) \\
\hline
Construct $Z'X$ & 0.0292s & 0.0461s &0.321s & 5311s\\
\hline
Solve SVD & 0.1052s & 2.5550s &0.4859s &9.118s \\
\hline
Total time & 0.1832s & 2.8975s & 112.6269s & 5781.856s \\
\hline
     \end{tabular}
\label{time}
\end{table*}

\section{Conclusion}
In this paper, we have introduced a computationally optimal randomized POD procedure for the extraction of ROMs for large scale systems such as those governed by PDEs. The ROM is constructed by perturbing the primal and adjoint system with Gaussian white noise, where the computational cost to construct the snapshot ensembles is saved when compared to perturbing the primal and adjoint system with impulses in BPOD/BPOD output projection algorithm. Also, it leads to a much smaller SVD problem, and an orders of magnitude reduction in the computation required for constructing ROMs  over the BPOD/ BPOD output projection procedure.  The computational results show that  the accuracy of the RPOD$^*$ is much more accurate than the BPOD output projection algorithm.

\appendices

\section{Proof of Proposition \ref{P1}}\label{AP1}
\begin{proof}
The adjoint snapshot ensemble $Z_r$ can be written as:
\begin{eqnarray}\label{snap_p2}
Z_r = U_{co} \beta_{co} + U_{c \bar{o}} \delta \beta_{c \bar{o}} + U_{\bar{c} o} \beta_{\bar{c} o} + U_{\bar{c} \bar{o}} \delta \beta_{\bar{c} \bar{o}},
\end{eqnarray}
where $\delta \beta_{c \bar{o}}  = \epsilon \beta_{c \bar{o}}$, $\delta \beta_{\bar{c} \bar{o}} = \epsilon \beta_{\bar{c} \bar{o}}$, $\epsilon$ is defined in assumption A4, $\beta_{c \bar{o}}, \beta_{\bar{c} \bar{o}}$ are coefficient matrices. 
From (\ref{snap_p1}) and (\ref{snap_p2}), 
\begin{eqnarray}
Z_r'X_r = \beta_{co}' \alpha_{co} + \delta \beta_{c \bar{o}}' \alpha_{c \bar{o}} +  \beta_{\bar{c} o} ' \delta \alpha_{\bar{c}o}  + \delta \beta_{\bar{c} \bar{o}}'  \delta \alpha_{\bar{c} \bar{o}} \nonumber \\
=  \beta_{co}' \alpha_{co} + \epsilon \underbrace{(\beta_{c \bar{o}}' \alpha_{c \bar{o}} +\beta_{\bar{c} o} ' \alpha_{\bar{c}o})}_{E_1} + O(\epsilon^2), \nonumber \\
= \beta_{co}' \alpha_{co} + \epsilon E_1 + O(\epsilon^2).
\end{eqnarray}

And similarly,
\begin{eqnarray}
Z_r' A X_r = \beta_{co}' \Lambda_{co} \alpha_{co} + \epsilon \underbrace{( \beta_{c \bar{o}}' \Lambda_{c \bar{o}} \alpha_{c \bar{o}} + \beta_{\bar{c} o}' \Lambda_{\bar{c} o} \alpha_{\bar{c} o} )}_{E_2} + O(\epsilon^2) \nonumber 
\end{eqnarray}
\begin{eqnarray}
=  \beta_{co}' \Lambda_{co} \alpha_{co} + \epsilon  E_2 + O(\epsilon^2).
\end{eqnarray}

There are ``$l$" controllable and observable modes, then rank $(Z_r'X_r) \geq l$ due to the small perturbations. Denote
\begin{eqnarray}
\bar{H}_r = \beta_{co}' \alpha_{co} = \bar{L}_r \bar{\Sigma}_r \bar{R}_r' + \bar{L}_o \bar{\Sigma}_o \bar{R}_o', \nonumber\\
{H}_r = Z_r' X_r =  \beta_{co}' \alpha_{co} + \epsilon E_1 = L_r \Sigma_r R_r' +  {L}_o {\Sigma}_o {R}_o, 
\end{eqnarray}
where $\bar{H}_r , H_r \in \Re^{n \times m}$, and WLOG, $n \leq m$.  

Here, $\bar{H}_r$ is the ideal Hankel matrix constructed with the simplifying assumption (assumption A3 is satisfied), and it can be seen that the true Hankel matrix  $H_r$ (assumption A4 is satisfied) can be viewed as adding a small perturbation of $\bar{H}_r$, i.e., $H_r = \bar{H}_r + \epsilon E_1$.

$\bar{\Sigma}_r \in \Re^{l \times l}$ contains the $l$ non-zero singular values of $\bar{H}_r$ and $(\bar{L}_r, \bar{R}_r)$ are the corresponding left and right singular vectors. $\bar{\Sigma}_o \in \Re^{(n-l) \times (n-l)} = 0$ are the rest singular values, and $(\bar{L}_o, \bar{R}_o)$ are the corresponding left and right singular vectors. Similarly, $\Sigma_r \in \Re^{l \times l}$ contains the first $l$ non-zeros singular values of $H_r$, and $\Sigma_o \in \Re^{ (n - l) \times (n - l)}$ contains the rest singular values. The left and right singular vectors are partitioned in the same way.

1). From the perturbation theory \cite{perturbation1,perturbation2}, the  perturbed singular values and singular vectiors $(\Sigma_r, L_r, R_r)$ are related to the singular values and singular vectors $(\bar{\Sigma}_r, \bar{L}_r, \bar{R}_r)$ as:
\begin{eqnarray}
\Sigma_r = \bar{\Sigma}_r + \epsilon E_3 + O(\epsilon^2), \nonumber \\
{L}_r = \bar{L}_r + \Delta L_r,  R_r= \bar{R}_r+ \Delta R_r,
\end{eqnarray}
where $E_3, \Delta L_r, \Delta R_r$ are some matrices, and  $\|\Delta L_r \|,$ $\| \Delta R_r \| \propto O(\epsilon)$. The expression of $E_3$ is given by the follows \cite{perturbation1}.

For $\bar{\sigma}_i \in \bar{\Sigma}_r,$ (strictly positive singular values) with multiplicity $t$, $\bar{L}_t, \bar{R}_t$ are the corresponding left and right singular vectors, the perturbed singular values ${\sigma}_i \in {\Sigma}_r$, and 
\begin{eqnarray} \label{nonzero}
{\sigma}_i = \bar{\sigma}_i + \epsilon \underbrace{\frac{\lambda_i (\bar{R}_t' E_1' \bar{L}_t + \bar{L}_t' E_1 \bar{R}_t)}{2}}_{e_i} + O(\epsilon^2), i = 1, \cdots, t
\end{eqnarray}
where $\lambda_i(.)$ denotes the $i^{th}$ eigenvalue of $(.)$. Thus,
\begin{eqnarray}
E_3 = \begin{pmatrix} e_1  & &  \\  & \ddots & \\&  & e_l  \end{pmatrix},
\end{eqnarray}
where $e_1, \cdots, e_l$ are the coefficients computed from (\ref{nonzero}).
 
2). The ROM $\tilde{A} = S_r A T_r$ can be proved to be a perturbation of $A_b$ in (\ref{bpodor}).

The POD bases $T_r, S_r$ are:
\begin{eqnarray}
T_r = X_r R_r \Sigma_r^{-1/2}, S_r = \Sigma_r^{-1/2} L_r' Z_r'. 
\end{eqnarray}

Therefore,
\begin{eqnarray}
\tilde{A} = S_r A T_r = {\Sigma}_r^{-1/2} {L}_r' Z_r' A X_r {R}_r {\Sigma}_r^{-1/2}  \nonumber \\
= {\Sigma}_r^{-1/2} {L}_r' (\beta_{co}' \Lambda_{co} \alpha_{co} + \epsilon E_2)  {R}_r {\Sigma}_r^{-1/2}. 
\end{eqnarray}

From (\ref{nonzero}), it can be proved that: 
\begin{eqnarray}
 \Sigma_r^{-1/2}  = \bar{\Sigma}_r^{-1/2}  + \epsilon C_r, 
\end{eqnarray}
where $C_r $ is a diagonal coefficient matrix. Therefore, 
\begin{eqnarray}
{\Sigma}_r^{-1/2} {L}_r' = (\bar{\Sigma}_r^{-1/2} + \epsilon C_r) (\bar{L}_r' + \Delta L_r') = \bar{\Sigma}_r^{-1/2} \bar{L}_r' + \Delta_1,
\end{eqnarray}
where $\Delta_1$ is some matrix, and  $\| \Delta_1 \|_2 = k_1 \epsilon$, $k_1$ is a constant.
Similarly,
\begin{eqnarray}
{R}_r {\Sigma}_r^{-1/2} = \bar{R}_r \bar{\Sigma}_r^{-1/2} + \Delta_2,
\end{eqnarray}
where $\Delta_2$ is some matrix, and $\| \Delta_2 \|_2 = k_2 \epsilon$, $k_2$ is a constant. Thus,
\begin{eqnarray}
\tilde{A} = \bar{\Sigma}_r^{-1/2} \bar{L}_r' \beta_{co}' \Lambda_{co} \alpha_{co} \bar{R}_r \bar{\Sigma}_r^{-1/2}  + \Delta_3 + O(\epsilon^2),
\end{eqnarray}
where $\Delta_3$ is some matrix, and $\| \Delta_3 \|_2 = k_3 \epsilon$, $k_3$ is a constant. If we let
\begin{eqnarray}
\bar{A} = \underbrace{\bar{\Sigma}_r^{-1/2} \bar{L}_r' \beta_{co}' }_{\bar{P}_r} \Lambda_{co} \underbrace{\alpha_{co} \bar{R}_r \bar{\Sigma}_r^{-1/2}}_{\bar{P}_r^{-1}},
\end{eqnarray}
then 
\begin{eqnarray}\label{eigper}
\tilde{A} = \bar{A} + \Delta_3 + O(\epsilon^2) = P_r  \Lambda_r  P_r^{-1}, 
\end{eqnarray}
Following the same proof in Section \ref{heuristic} ( (\ref{bpodor})-(\ref{bpodeig})), it can be proved that $\Lambda_{co}$ are the eigenvalues of $\bar{A}$, and $\bar{P}_r$ are the corresponding eigenvectors.  

3). From perturbation theory \cite{eigp}, since
\begin{eqnarray}\label{eigper}
\tilde{A} = \bar{A} + \Delta_3 + O(\epsilon^2) = P_r  \Lambda_r  P_r^{-1}, 
\end{eqnarray}
it can be proved that $\| P_r - \bar{P}_r \| \leq \| \Delta_3 \| = k_3 \epsilon, \| \Lambda_r - \Lambda_{co} \| \leq \| \Delta_3 \| = k_3 \epsilon$.   Thus,
\begin{eqnarray}\label{bases_pert}
\Psi_{r} = T_r P_r = X_r \bar{R} _r \bar{\Sigma}_r^{-1/2} \bar{P}_r  + \Delta_4, \nonumber \\
\Phi_{r} = P_r^{-1} S_r =  \bar{P}_r^{-1} \bar{\Sigma}_r^{-1/2} \bar{L}_r' Z_r'  + \Delta_5,
\end{eqnarray}
where $\Delta_4, \Delta_5$ are some matrices and $\| \Delta_4 \|, \| \Delta_5 \| \propto O(\epsilon)$. Substitute (\ref{bases_pert}) into the ROM Markov parameters, and collect the small perturbation terms, we have:
\begin{eqnarray}
C_r A_r^i B_r = C \Psi_r {\Lambda_{r}^i} \Phi_{r} B 
= C V_{co} \Lambda_{co}^i U_{co}' B' + \Delta,
\end{eqnarray}
where $\Delta$ is some matrix, and $\| \Delta \|  \propto O(\epsilon )$. 

\section{Proof of Corollary \ref{cr2}}\label{AP2}
For $\bar{\sigma}_i \in \bar{\Sigma}_o$ (zero singular values) \cite{perturbation1} with multiplicity $n-l$, the corresponding left and right singular vectors are $\bar{L}_o, \bar{R}_o$. The perturbed singular values ${\sigma}_i \in {\Sigma}_o$ and
\begin{eqnarray}\label{zero}
{\sigma}_i = \epsilon \sqrt{\lambda_i (\bar{R}_o' E_1' \bar{L}_o \bar{L}_o' E_1 \bar{R}_o)}, i = 1, \cdots n - l.
\end{eqnarray}

From (\ref{nonzero}) and (\ref{zero}), we can see that:
\begin{eqnarray}
\sigma_l = \bar{\sigma}_l + e_l \epsilon + O(\epsilon^2),  
\sigma_{l+1} = e_{l+1} \epsilon, 
\end{eqnarray}
Hence,we have:
\begin{eqnarray}
 \sigma_{l+1} \propto O(\epsilon),
\end{eqnarray}
and
\begin{eqnarray}\label{rme}
\| C_r A_r^i B_r - C A^i B \| \propto O(\epsilon) \propto O(\sigma_{l+1}). 
\end{eqnarray}
\end{proof}
\bibliographystyle{IEEEtran}
\bibliography{IPOD_refs}

\end{document}